\documentclass[12pt,a4paper,twoside]{article}
\usepackage{amsfonts,amsmath,amsthm,euscript,mathrsfs,dsfont,bm}
\usepackage[english]{babel}
\usepackage{geometry}
\usepackage{graphicx}
\usepackage{subfigure}
\usepackage{indentfirst} 
\usepackage{multirow, multicol}	
\usepackage{float}	
\usepackage{hyperref}
\hypersetup{
    colorlinks=true,
    linkcolor=black,
    filecolor=magenta,
    urlcolor=cyan,
}
\begin{document}


\theoremstyle{plain}
\newtheorem{thm}{Theorem}[section]
\newtheorem{cor}{Corollary}[section]
\newtheorem{property}{Property}[section]
\newtheorem{prop}{Proposition}[section]
\newtheorem{lemma}{Lemma}[section]
\theoremstyle{definition}
\newtheorem{exa}{Example}[section]
\newtheorem{df}{Definition}[section]
\newtheorem{rmk}{Remark}[section]
\newcommand{\Tspace}{\rule{0pt}{3ex}}
\newcommand{\Bspace}{\rule[-1.2ex]{0pt}{0pt}}
\renewcommand{\theequation}{\thesection.\arabic{equation}}
\renewcommand{\thetable}{\thesection.\arabic{table}}
\renewcommand{\thefigure}{\thesection.\arabic{figure}}
\renewcommand{\arraystretch}{1.2}
\long\def\symbolfootnote[#1]#2{\begingroup%
\def\thefootnote{\fnsymbol{footnote}}\footnote[#1]{#2}\endgroup}
\numberwithin{table}{section}
\numberwithin{figure}{section}
\renewcommand{\qed}{\begin{flushright}\vspace{-.5cm}$\Box$\end{flushright}}

\newcommand{\NorT}{\ensuremath{\mathcal{N}\mathcal{T}}}
\newcommand{\Nor}{\ensuremath{\mathcal{N}}}
\newcommand{\cov}{\mathrm{Cov}}
\newcommand{\corr}{\mathrm{Corr}}
\newcommand{\C}{\mathbb{C}}
\newcommand{\E}{\mathbb {E}}
\newcommand{\F}[1]{\mathcal{F}_{#1}}
\newcommand{\iten}[1]{\vspace{.2cm}\noindent{\bf #1}}
\newcommand{\K}{\mathcal{K}}
\renewcommand{\L}{\ell}
\newcommand{\LL}{\mathcal{B}}
\newcommand{\N}{\mathbb {N}}
\newcommand{\nn}[1]{\mbox{\boldmath{$#1$}}}
\newcommand{\noise}{$\{Z_t\}_{t \in \mathbb{Z}}{}$}
\newcommand{\prob}{\mathbb{P}}
\newcommand{\pe}{$\{X_t\}_{t \in \mathbb{Z}}{}$}
\renewcommand{\proof}{\noindent\textbf{Proof: }}
\renewcommand{\qed}{\begin{flushright}\vspace{-.5cm}$\Box$\end{flushright}}
\newcommand{\R}{\mathbb {R}}
\newcommand{\stg}{$\left\{X_{t}\right\}_{t=1}^{n}{}$}
\newcommand{\var}{\mathrm{Var}}
\newcommand{\Z}{\mathbb {Z}}
\newcommand{\mat}[1]{\mbox{\boldmath{$#1$}}}
\renewcommand{\qed}{\begin{flushright}\vspace{-.5cm}$\Box$\end{flushright}}
\newcommand{\fim}{\hfill $\Box$}
\newcommand{\stk}[1]{\stackrel{#1}{\longrightarrow}}
\newcommand{\s}{\hphantom{x}}
\newcommand{\Ind}[2]{\ensuremath{\mathbb {I}_{#1}{\mbox{\footnotesize $\left(#2\right)$}}}}
\newcommand{\B}{\mathcal{B}}
\newcommand{\si}{\mathrm{sign}} 
\thispagestyle{empty}

\vskip2cm
{\centering
\Large{\bf Decision Theory and Large Deviations for Dynamical Hypotheses Tests:
The Neyman-Pearson, Min-Max and Bayesian Tests}\\
}
\vspace{1.0cm}
\centerline{\large{\bf{Hermes H. Ferreira, Artur O. Lopes  and S\'ilvia R.C. Lopes}}}

\vspace{0.5cm}

\centerline{Mathematics and Statistics Institute}
\centerline{Federal University of Rio Grande do Sul}
\centerline{Porto Alegre, RS, Brazil}

\vspace{0.5cm}

\centerline{\today}

\begin{abstract}
\noindent

\noindent We analyze hypotheses tests using classical results on large deviations to compare two models, each one described by a different H\"older Gibbs probability measure. One main difference to the classical hypothesis tests in Decision Theory is that here the two measures are singular with respect to each other. Among other objectives, we are interested in the decay rate of the wrong decisions probability, when the sample size $n$ goes to infinity.  We show a dynamical version of the Neyman-Pearson Lemma displaying the ideal test within a certain class of similar tests. This test becomes exponentially better, compared to other alternative tests, when the sample size goes to infinity.  We are able to present the explicit exponential decay rate. We also consider both, the Min-Max and a certain type of Bayesian hypotheses tests. We shall consider these tests in the log likelihood framework by using several tools of Thermodynamic Formalism. Versions of the Stein's Lemma and Chernoff's information are also presented.

\vspace{0.3cm}
\noindent \textbf{Keywords:} Decision Theory,  Large Deviations Properties, Rejection Region, Ney\-man-Pearson Lemma, Min-Max Hypotheses Test, Bayesian Hypotheses Test, Thermodynamic Formalism, Gibbs Probabilities.
\vspace{.2cm}
\newline
\noindent \textbf{2020 Mathematics Subject Classification}: 62C20, 62C10, 37D35.

\end{abstract}

\section{Introduction}\label{Intro}
\renewcommand{\theequation}{\thesection.\arabic{equation}}
\setcounter{equation}{0}

The problem we are interested in this work can be simply expressed as the following: there are two ergodic probability measures $\mu_0$ and $\mu_1$ that we know, in advance, what they are. A data set is obtained by ergodicity, but we do not know, in advance, if it was originated from $\mu_0$ or $\mu_1$. Suppose it comes from the random process associated with $\mu_1$. From this data set, we need to decide which one of the two processes generated this data set. A hypotheses test is a method that helps us to make such a choice. We want to design tests such that by observing the data set generated from the random process - indexed by the integers - we shall be able to make the right decision, that is, to choose the alternative $\mu_1$. In addition, we wish to make optimal choices among possible tests. Classical results on Large Deviations (LD) properties can estimate the risk that a wrong decision is made. From the Bayesian point of view, we should attach to $\mu_0$ a probability $\pi_0$ and to $\mu_1$ a probability $\pi_1$, such that $\pi_0+\pi_1=1$.

We shall consider a certain class of probabilities on the Borel sigma algebra of the symbolic space $\Omega=\{1,2,...,d\}^\mathbb{N}$.  Points in $\Omega$ will be denoted by $y=(b_0,b_1,b_2,b_3,\cdots,b_n,\cdots) $. The shift transformation $\sigma$ is given by $ \sigma(b_0,b_{1},b_2,b_3,\cdots,b_n,\cdots) = (b_1,b_2,b_3,\cdots,b_n,\cdots)$.

Our main interest lays on the decay rate of the probability of wrong decisions in terms of the sample size $n$. Using large deviation techniques we will try to determine the optimal rate for different cases.

We shall extend the reasoning described on page 91, section VI, of \cite{Buck}, where the author considers LD properties. However, we point out that, in \cite{Buck}, there is no dynamics involved in the process.

A reference for Thermodynamic Formalism is \cite{PP} (see also \cite{Denk1} and \cite{L3}).  For results on Large Deviations for Thermodynamic Formalism we refer the reader to  \cite{Kif}, \cite{L4}, \cite{L2} and  \cite{L3}. Important references for basic results in Hypotheses tests are \cite{AR}, \cite{Buck}, \cite{CLP}, \cite{Denk}, sections 3.4 and 3.5 in \cite{DZ}, \cite{KLS}  and \cite{Roh}. For additional results on the Bayesian point of view in Thermodynamic Formalism, we refer the reader to \cite{ELLM}, \cite{LLV} and \cite{Nobel1}.

References on applications of hypotheses tests to Game Theory are \cite{Barni}, \cite{Bla}.

Invariant probabilities for the shift transformation correspond to stationary processes $X_n$, for $n \in \mathbb{N}$, with values on $\{1,2,\cdots,d\}$.

Given a H\"older potential $A: \Omega \to \mathbb{R}$ {\it the pressure of $A$} is defined as
\begin{equation*}
P(A) := \sup_{\mu \,\, \emph{invariant for the shift}} \left\{ \int A d \mu + h(\mu)\right\},
\end{equation*}
where $h(\mu)$ is  \emph{the Shannon-Kolmogorov entropy} for the invariant probability measure $\mu$. The unique probability which realizes such supremum is called the H\"older equilibrium probability for the potential $A$. It's known that $P(A)$ is an analytic function on the potential $A$ (see \cite{PP}). This property is quite useful to obtain good large deviation properties (see, for instance, \cite{DZ}, \cite{Ellis}, \cite{L2} and \cite{L3}).

Consider a H\"older continuous function $\log J: \Omega\to \mathbb{R}$, where $J>0$, such that
\begin{equation} \label{kle}\sum_{a=1}^d J (a, b_{0},b_1, b_2,b_3, \cdots)=1,
\end{equation}
for all  $y=(b_0,b_{1},b_2,b_3,\cdots) \in \Omega$.

In this case, $P( \log J)=0$ and the H\"older  equilibrium probability will be called a {\it  H\"older  Gibbs equilibrium probability for} $\log J$.
For instance, when $\log J = - \log d$, the corresponding equilibrium probability will be the maximum entropy probability, which is the independent probability with weights $1/d$.

Equilibrium probabilities play a central role in several problems in Statistical Physics and Information Theory. Hypotheses tests are relevant in all these domains;
see, for instance \cite{Beno}, \cite{Sagawa}, \cite{Suhov} and \cite{von}.
 For instance,  \cite{van} describes the role of hypotheses tests in the Discovery of a Higgs Boson. The scope of our work can be eventually useful in all  these areas.

The H\"older  equilibrium probabilities $\mu$ have no atoms (no points $x\in \Omega$ such that $\mu(x)>0$).
For each H\"older Gibbs probability $\mu$, one can associate a unique H\"older continuous function $\log J: \Omega=\{1,2, \cdots,d\}^\mathbb{N}\to \mathbb{R}$ (see \cite{PP}).
We call  $J$ the \emph{Jacobian} of $\mu$. Any Jacobian function considered here is of H\"older class and from this follows that the associated probability measure $\mu$ is ergodic (see \cite{PP}).

A stationary Markov process  $X_n$, $n \in \mathbb{N}$, with state space $\{1,2,\cdots, d\}$, associated to a $d \times d$ line stochastic matrix ${\cal P}$, defines a shift invariant probability $\mu $ in $\{1,2, \cdots,d\}^\mathbb{N}.$ One can show that
the associated Jacobian function $J$ on the cylinder $\overline{i\,j}$ has the constant value $\frac{\pi_i\,p_{i j}}{\pi_j}$, where $\pi=(\pi_1,\pi_2,\cdots,\pi_d)$ is the initial stationary vector for  $\mathcal{P}$.  In the $2 \times 2$ case, we get that $\frac{\pi_i\,p_{i j}}{\pi_j}=p_{j\,i}$, for $i,j=1,2$ (see Section \ref{Example}). The Jacobian is the natural extension of the concept of stochastic matrix (see example 1, in \cite{L3} or page 27 in \cite{PP}). The references \cite{Baha}, \cite{GR} and \cite{GLR} consider statistical tests for Markov Chains but using different methodologies.

The Shannon-Kolmogorov entropy of such shift invariant probability $\mu$ is given by the formula $h(\mu)= - \int \log J\, d \mu.$

 When considering two probabilities $\mu_0$ and $\mu_1$, with respective Jacobians $J_0$ and $J_1$, it is natural to consider the log-likelihood
 \begin{equation} \label{pesc}
 \log J_0 - \log J_1.
\end{equation}
In this case, a relation that will be important in future sections (see expressions \eqref{lol} and \eqref{lol1}) is:
 \begin{equation} \label{pesc1}
- h(\mu_0 ) - \int \log J_{1\,} \,d \mu_{0} = \int (\log J_{0} - \log J_{1\,}) \,d \mu_{0} \,>0 .
\end{equation}
In the reference \cite{LLV}, such kind of loss function is also analyzed under time evolution  but from another perspective. Expression \eqref{pesc} is a particular case of a more general class of loss functions described by expression (3) in \cite{Nobel1}.

The value $\int (\log J_{0} - \log J_{1\,}) \,d \mu_{0} \,$ is also known as \emph{Kullback-Leibler divergence}. For some results on Kullback-Leibler divergence,
we refer the reader to \cite{Cha}, \cite{GR}, \cite{LM}, \cite{LR} and \cite{Sagawa}.

Two  different H\"older Gibbs probabilities are  singular to each other (see \cite{PP} or appendix 7.3 in \cite{L3}). We point out that, in most of the cases in the classical Hypotheses tests setting, it is considered families of probabilities which are {\bf absolutely continuous} with respect to each other. The set of H\"older Gibbs probabilities is dense (in the weak$*$ topology) in the set of shift  invariant probabilities (see, for instance, \cite{L2}). Therefore, we are considering here a larger class of  possible pairs.  It is the information of the Jacobians that allows us to compare two probabilities which are   singular with respect to each other.

Here we will just consider probabilities $\mu$ on $\Omega$ of H\"older Gibbs type.
The associated stochastic process $\{X_n\}_{n \in \N}$, taking values on $\{1,2,\cdots,d\}$, is described by
\begin{equation}\label{cyll}
\prob(X_{0}=a_{0},X_2=a_2,\cdots,X_ n=a_n) = \mu (\overline{a_{0},a_2,\cdots,a_n}),
\end{equation}
where $\overline{a_{0},a_2,\cdots,a_n}\subset \Omega$ is a general cylinder set.

Our tests are based on estimations of Birkhoff  sums, $y \in \Omega$,
\begin{equation} \label{smart}
\frac{1}{n} \sum_{i=0}^{n-1} \log\left(\frac{J_{0}(\sigma^i (y))}{J_{1\,}(\sigma^i (y))} \right),
\end{equation}
that is, via the Jacobians of $\mu_0$ and $\mu_1$, and rejection regions given by \eqref{yokon} and \eqref{yuo}.
We say the sample size goes to infinity when we take the limit $n\to \infty$ in expression  \eqref{smart}. We point out that the terminology {\it sampling} in  generally refers to the case where the sample is obtained via an independent process (which is not the case here).

  Some of the results  in \cite{LLV} for Bayes posterior convergence are based on the measure of cylinder sets.  Here the dynamical tests are based on the Jacobians.

The paper is organized as follows: in Section \ref{HypoTest} we present the basic idea of two simple hypotheses test in the thermodynamical formalism sense, where the definition of the \emph{type I} and \emph{type II errors} are stated. In Section \ref{LDP}, Large Deviation properties and some basic results are presented. The Neyman-Pearson Lemma
is the  main result in Section \ref{NP}. The Min-Max hypotheses test
is presented in Section \ref{MinMax}, while the Bayesian hypotheses test is in Section \ref{Bayes}. Finally, Section \ref{Example} presents an example based on the Min-Max hypotheses test.

We point out that when showing  optimality  in Section \ref{NP} (the claim of Theorem A), when $n \to \infty$,  for our version of the dynamical Neyman-Pearson Lemma, we compare a test with others in the same class
(and not with respect to the universe of all possible tests). However, we will be able to present the explicit coefficient of  exponential  decay to zero of wrong decisions, when $n \to \infty$, in terms of the
Legendre transform of the Pressure function.

\section{Preliminaries on Hypotheses Tests}\label{HypoTest}
\renewcommand{\theequation}{\thesection.\arabic{equation}}
\setcounter{equation}{0}

In this section, we set the preliminaries and basics concepts for introducing simple hypotheses test in the setting of thermodynamic formalism.

For each  stochastic process  $\{X_n\}_{n \in \N}$ we  consider the associated  probability on the symbolic space $\Omega=\{1,\cdots,d\}^\N$. We can test two simple hypotheses in the following way:
\vspace{0.1cm}
\begin{enumerate}
\item[] \textbf{$H_0$:} $\{X_n\}_{n \in \N}$ is described by $\mu_0$ with Jacobian $J_0$
\item[] \textbf{$H_1$:} $\{X_n\}_{n \in \N}$ is described  by $\mu_{1}$ with Jacobian $J_1$.
\end{enumerate}

\vspace{0.2cm}

The two measures $\mu_0$ and $\mu_1$ considered here are H\"older Gibbs probabilities and, therefore, are singular with respect to each other. As far as the authors know, this
type of tests for a pair of singular probabilities was not considered before in the literature.

We want to decide which one of the two hypotheses is true from a data set of the form
$y_i= \sigma^i (y_{0})$, $i=0,1,2,\cdots,n-1$, where $y_{0}\in \Omega$ is chosen at random according to a given ergodic measure $\mu$. In Section \ref{NP} and Subsection \ref{Optimal}, we will choose to fix such $\mu$ as $\mu_{1}$.
We are interested in the Large Deviations properties for such types of tests.
The information of the probabilities $\mu_0$ and $\mu_1$ obtained  from the respective Jacobians $J_0$ and $J_1$ allows us to consider tests of the log-likelihood type. More precisely, it will be natural to consider loss functions of the form $\log J_0 - \log J_1$, under the time evolution of certain stochastic processes.

One can announce $H_1$ when, in fact, $H_0$ is true. This is called the \emph{false alarm} or  \emph{type I error}. The probability of false alarm is usually denoted by $\alpha$, which is called \emph{the test size}. Therefore, the value $\alpha$ denotes
 $ \alpha:=\prob(\mbox{Decide}\, H_1\, | H_0 \, \mbox{is true})$. We choose
$\alpha$ such that $0<\alpha<1$.

On the other hand, one can announce $H_0$ when, in fact, $H_1$ is true. This is called a \emph{misspecification} or  \emph{type II error}. The  probability of misspecification is usually denoted by $1-\beta$. The \emph{detection rate} is the value $\beta \in (0,1)$, which describes the probability $\beta:=\prob(\mbox{Decide}\, H_1 \,\,| \,\,  H_1 \, \mbox{is true})$. The value $\beta$ is called the \emph{power of the test}. In general, one hopefully would like  to fix a value of $\beta$ close to $1$.

We do not know in advance which hypothesis $H_0$ or $H_1$ is more likely to happen (at least for non-Bayesian tests).

In Section \ref{NP}  we are interested in an optimal procedure which can be described schematically in the following way:  we  consider a class of possible tests and one of them will be fixed for comparison. The other tests will be called alternative tests. We shall investigate which one is the test that minimizes the total error probability, that is, we would like to maximize the power of the test, among all tests that have a certain test size (in the sense of appendix B1, on page 611, in \cite{von}).  This result can be seen as a dynamical version of  the  Neyman-Pearson Lemma.

We point out that the Bayesian point of view  will be explored in Section \ref{Bayes}.

The \emph{test statistics} we consider will be associated to the time series  $y, \sigma (y),\cdots, \sigma^n (y)$ and to the log-likelihood form given by

\begin{equation} \label{sma}
S_n (y):=\frac{1}{n} \sum_{i=0}^{n-1} \log\left(\frac{J_{0}(\sigma^i (y))}{J_{1\,}(\sigma^i (y))} \right).
\end{equation}
A similar log-likelihood  test was considered in section VI of \cite{Buck}, but no dynamics. Here $y\in \Omega$ will be taken in a Birkhoff's set for an ergodic probability (not sampling).

\begin{rmk} \label{rrtu}
Given the probabilities $\mu_0$ (associated to $H_0$) and  $\mu_1$ (associated to $H_1$), as $\int   (\log J_{0 } - \log J_{1\,}) \,d \mu_{1} \,<0$, then,
for $\mu_1$ almost every $x$, we get that $S_n$ will become negative.
On the other hand, as $\int   (\log J_{0 } - \log J_{1\,}) \,d \mu_{0} \,>0$,
then, for $\mu_0$ almost every $x$, $S_n$ will become positive.  This leads us to the  study and design of tests via rejection regions.
\end{rmk}
We shall introduce a sequence $u_n$, $n \in \N$, which will be necessary for the test.
The rejection region $\mathcal{R}_n$ is defined as
\begin{equation} \label{yokon}
\mathcal{R}_n:=\{x\in \Omega \,|\, S_n< u_n\}, \ \ n \in \mathbb{N}.
\end{equation}
We assume that the sequence $u_n$ has a limit:
\begin{equation} \label{yuo}
\lim_{n \to \infty} u_n = E.
\end{equation}

Without loss of generality, we  assume that the sequences $u_n$, $n \in \mathbb{N}$,
are monotonous. The values of $E$ will be taken in an open interval which will be described later by Remark \ref{supre}. In the dynamic sense, the important quantity is the limit value $E$ and not the specific values $u_n$.

Given a sample of size $n$, if $S_n<u_n$, we announce $H_{1}$, when $H_0$ is true, and if $S_n>u_n$ we announce $H_{0}$, when $H_1$ is true. This is the interpretation of wrong decisions. In all tests considered here, the main point is to find the optimal choice of $E$ (which is  the limit of $u_n$,  $n \in \mathbb{N}$) and its relationship with the asymptotic
values of
\begin{equation} \label{jjj}
\mu_{1}\left( S_n\geq  u_n \right)\ \ \mbox{ or }\ \ \mu_{0}\left( S_n\leq  u_n \right),
\end{equation}
for large $n$.

For all tests, we shall consider large data sets and we are interested in minimizing the exponential rate of the probability of a wrong decision. In this direction, it will be necessary to study Large Deviations properties first. The results on Large Deviations, that we need here, shall be presented in Section \ref{LDP}.

From an experimental point of view, the confidence in the information $S_n (y)> u_n$, given by  a Birkhoff's time series  of large size $n$ obtained from the ergodicity of  $\mu_0$,  is related to the asymptotic velocity of the sequence  $\mu_{0}\left( S_n\leq  u_n \right)$ going to zero, when $n$ goes to infinity.

For the Neyman-Pearson Lemma (see Section \ref{NP} and Subsection \ref {Optimal}),
the large samples will be taken according to $\mu_{1}$ and we want to estimate how small is the probability $1 -\beta_n$ of announcing $H_0$ when $H_1$ is true.
In this case, we shall consider samples of the process $S_n$, for $n \in \N$, which will be produced with the random choice given by $\mu_{1}$ and not by $\mu_{0}$.

We denote by $\beta_n := \mu_1 ({\cal R}_n)$  \emph{the power of the test} at time $n$.
We want to analyze the \emph{misspecification probability}, or the \emph{type II error}, which will be denoted by
\begin{equation} \label{bbr}
1-\beta_n= \mu_1( \Omega  -{\cal R}_n)= 1 -\mu_1({\cal R}_n)= \mu_1( S_n > u_n),
\end{equation}
from samples of size $n$.

The asymptotic values of the probabilities of $\mu_1(S_n>u_n)$ and  $\mu_0(S_n\leq u_n)$,
for $n \in \mathbb{N}$, are the essential information we shall consider. The main issue here is: $\mu_1 \{x\,|\, S_n>  u_n\}$ is associated with a wrong decision by announcing
$H_0$ when $H_1$ is true. On the other hand $\mu_0
\{x\,|\, S_n\leq  u_n\}$ is associated with another type of wrong decision by
announcing $H_1$ when $H_0$ is true.

Alternative tests $A$ to the $NP$ test are associated to sequences $\tilde{u}_n$, $n \in \mathbb{N}$, and rejection regions
\begin{equation} \label{yoko567}
\mathcal{R}_n^A:=\{x\in \Omega \,|\, S_n< \tilde{u}_n\}, \ \ n \in \mathbb{N}.
\end{equation}

Given a probability $\mu_1$ and a test $A$ associated to a sequence
$\tilde{u}_n$ and $E\leq 0$, in the case there exists $c>0$ such that
$$
\lim_{n \to \infty} \frac{1}{n} \log (\,\mu_{1} ( \Omega - \mathcal{R}_n^A) )=\lim_{n \to \infty} \frac{1}{n} \log (\,\mu_{1} ( S_n> \tilde{u}_n)\,) <  c<0,
$$
then the test works fine (see Remark \ref{rrtu}). The value $c$ regulates the speed of exponential convergence to $0$. This leads us to the design of optimal tests among the possible $A$ tests.

The main result  of Section \ref{NP}, related to the maximization of the test power, is given by Theorem $4.1$. It is a dynamical version of the Neyman-Pearson Lemma. We state this theorem below.

{\bf In order to simplify the notation we will consider the specific case
where
\begin{equation} \label{yuo32}
\lim_{n \to \infty} u_n = E^{NP}=  \int \log J_0 \,d \mu_{0} \,-   \int \log\,J_1\,\mathrm{d}\mu_{0}.
\end{equation}
}

\vspace{0.35cm}

\noindent {\bf Theorem A.} \label{aa} \emph{Consider the test when}
\begin{equation*}
u_n \to E^{NP}=  \int \log J_0 \,d \mu_{0} \,-   \int \log\,J_1\,\mathrm{d}\mu_{0}.
\end{equation*}
\emph{ When comparing test for the above value of $E^{NP}$,  with the  tests for other possible values of  $E$, the one associated to $E^{NP}$ minimizes the type II error.}

\
\emph{ The above means that for any other alternative test $A$, that is, for any  sequence $\tilde{u}_n$, converging to a certain value $E$, we get}
$$
\lim_{n \to \infty} \frac{1}{n} \log (\,\mu_{1} ( S_n> \tilde{u}_n)\,) >  \lim_{n \to \infty} \frac{1}{n} \log (\,\mu_{1} ( S_n> u_n)\,)
$$
$$
=\,-\left(\int \log J_0 \,d \mu_{0} \,-   \int \log \,J_1\,\,d\mu_0\,\right)\, =\,-  E^{NP},
$$
\emph{for tests $A$ satisfying the following condition concerning  the test size:}
$$\mu_{0}\left(S_n<\tilde{u}_n \right)\leq \mu_{0}\left(S_n<u_n \right).$$

\smallskip

 One can show that the test associated with $E{NP}$ is exponentially better than the others, when $n\to \infty$, via the explicit decay rate  expressions \eqref{orty} and  \eqref{expp}. One could also fix another value $\overline{E}$ and get  similar optimality results when comparing with other test for different values of $E\neq \overline{E}$.

The optimality of the claim  of Theorem A can be expressed by saying that
$$
\lim_{n \to \infty} \frac{1}{n} \log (\,\mu_{1} (\mathcal{R}_n^A)\,) >  \lim_{n \to \infty} \frac{1}{n} \log (\,\mu_{1} (\mathcal{R}_n^{NP})),
$$
where $\mathcal{R}_n^{NP}$ is the rejection region for $u_n \to E^{NP}$.
\smallskip

\smallskip

In the other two tests (see Sections \ref{MinMax} and  \ref{Bayes}) we shall also consider large samples $S_n$, but we have to compare the corresponding asymptotic laws according to  $\mu_1$ and also $\mu_0$, in terms of the expression \eqref{jjj}.

Section \ref{MinMax} shall consider loss functions and the Min-Max hypotheses test.
In that case,  it will be natural to consider the \emph{pressure} as a function of a real parameter $t\in \mathbb{R}$, more precisely, we shall need the function $P_1$ given  by
$$
t\to P_1(t) := P(\,t\,( \log J_0- \log J_1) + \log J_1).
$$

The main result in Section \ref{MinMax} is Theorem \ref{eer} that we state here:

\vspace{0.35cm}

\noindent {\bf Theorem B.} \emph{In the Min-Max hypotheses test, the best choice of $E$ will be $E=0$. Moreover, the best decay rate for minimizing the probability of wrong decisions is
given by $e^{n \,r}$, where $r<0$ is the minimum value of the pressure function $P_1$. The value of $r$ is given by $r=h(\mu)+ \int \log J_1 d \mu$, where $\mu$ is a Gibbs probability measure satisfying $\int (\log J_0 - \log J_1) d \mu=0.$}

\vspace{0.35cm}

In Section \ref{Bayes} a certain type of Bayesian hypotheses test will be studied. Hypothesis $H_0$ will have probability $\pi_0$ and hypothesis $H_1$ will have probability $\pi_1$, where $\pi_0+ \pi_1=1$. In this section, we shall consider rejections regions of the form
\begin{equation} \label{yokono}
\mathcal{R}_{n,\lambda}:=\left \{x\in \Omega \,\bigg|\, \frac{1}{n} \sum_{i=0}^{n-1} \log J_{\lambda}(\sigma^i (y))   < u_n\right\},\ \ \mbox{for} \ \ n \in \mathbb{N},
\end{equation}
where

\begin{equation} \label{sac0}
J_\lambda := \lambda J_ 1 + (1-\lambda) J_0, \ \ \mbox{for}\ \ \lambda\in [0,1].
\end{equation}

We shall estimate   $\pi_1\,\mu_1 (S_n > u_n)$ and  $\pi_0 \,\mu_0 (S_n \leq u_n)$,
for $n \in \mathbb{N}$. For this test, we shall exhibit the best value of $E_\lambda$ that minimizes the probability of a wrong decision (see expressions \eqref{cross}, \eqref{relentropy} and \eqref{xcafe4}), for each $\lambda$. We shall also find the best possible $E_\lambda$ value, producing the best decay rate, among all possible values of $\lambda$.

We will show  a version of Chernoff's information in Section \ref{Bayes}.

\section{Preliminaries on Large Deviations Properties} \label{LDP}

In this section, we shall present the specific Large Deviations properties which will be necessary for the proof of our main results in Sections \ref{NP}, \ref{MinMax}, and \ref{Bayes}. First, we briefly mention some basic results on  the topic of  large deviations  in Themodynamic Formalism. Denote by $\mu$ a H\"older Gibbs probability on $\Omega$  and consider a H\"older function $\varphi:\Omega \to \mathbb{R}$. The Birkhoff's Theorem claims that,  for $\mu$-almost every point $y \in \Omega$,
we get that $\lim_{n\to \infty} \frac{1}{n} \sum_{i=0}^{n-1} \varphi (\sigma^i (y))= \int \varphi  d\mu$. We are interested in the estimation of  the deviations of $\frac{1}{n} \sum_{i=0}^{n-1} \varphi (\sigma^i (y))$ (the times series mean value) from the limit $\int \varphi  d\mu$. The large deviation function $I:\mathbb{R} \to [0,\infty)$ will help in
this direction. The function $I$ will be analytic and shall take the zero value only in the point $\int \varphi d \mu$.  The function $I$ is the Legendre transform of the analytic strictly convex  function
\begin{equation}\label{phili}
t \to c(t):= \lim_{n \to \infty} \frac{1}{n}\log\int  e^{t\,\, \sum_{i=0}^{n-1}   \varphi (\sigma^i (y)) } d \mu (y).
\end{equation}

Under the above hypotheses the following  large deviation results are true (see, for instance, \cite{Kif}, \cite{L4} or \cite{L3} for more details):

\begin{enumerate}
\item given an open  interval $(a,b)\subset \mathbb{R}$,
\begin{equation}\label{tuty1}
\lim_{n \to \infty} \frac{1}{n} \log \mu \left\{\, y\in \Omega \,|\,  \frac{1}{n} \sum_{i=0}^{n-1} \varphi (\sigma^i (y))\in (a,b)\,\right\} \geq  -\, \inf \{ I(z) \,|\, z \in (a,b)\}.
\end{equation}

\item given a closed  interval $[a,b]\subset \mathbb{R}$,
\begin{equation}\label{tuty2}
\lim_{n \to \infty} \frac{1}{n} \log \mu \left\{\,  y\in \Omega \,|\,  \frac{1}{n} \sum_{i=0}^{n-1} \varphi (\sigma^i (y))\in [a,b]\,\right\} \leq  -\, \inf \{ I(z) \,|\, z \in [a,b]\}.
\end{equation}
\end{enumerate}

The function $I$ depends on $\mu$ and on $\varphi$. The probability of the set of points $x$, such that   $\frac{1}{n} \sum_{i=0}^{n-1} \varphi (\sigma^i (y)) $ deviates $\epsilon>0$ from the mean value $\int \varphi d \mu$, converges exponentially to zero, since
$$
\lim_{n \to \infty} \frac{1}{n} \log \mu \left\{\,  y\in \Omega \,|\,  \frac{1}{n} \sum_{i=0}^{n-1} \varphi (\sigma^i (y))\in (-\infty, -\epsilon]\,\cup [\epsilon,\infty) \right\} =$$
\begin{equation}\label{tutyte}
= -\, \inf \{ I(z) \,|\, z \in (-\infty, -\epsilon]\,\cup
[\epsilon,\infty)\}<0.
\end{equation}

\begin{enumerate}
\item Given the sequence $a_n$ and $c\in \mathbb{R}$, we say that $a_n \sim e^{ -n\, c}$ if
$$
\lim_{n \to \infty} \frac{1}{n} \log a_n =- c.
$$

\item{ }
The value $-c$ regulates the speed of the exponential convergence to zero
\end{enumerate}
In order to capture the meaning and usefulness of the  above bounds in expressions \eqref{tuty1} and \eqref{tuty2}, we can say, in a simplified way, that
\begin{equation}\label{estimatetu}
\mu \left\{\, y\in \Omega \,|\,  \frac{1}{n} \sum_{i=0}^{n-1} \varphi (\sigma^i (y))\in (a,b)\, \right\} \sim e^{ - \,n\,  \inf \{ I(z) \,|\, z \in (a, b) \}}.
\end{equation}

It will not make much difference in our arguments if the intervals are of the form $(a,b)$ or $[a,b]$. If $\int \varphi  d\mu$ is not in $[a,b]$ then \eqref{estimatetu} converges exponentially  to zero with $n$.

%

For the study of large deviations on the topic of log-likelihood tests, we shall be  interested in estimating
\begin{equation}\label{estimate}
\mu_{1}\left( S_n> u_n \right)= \mu_{1}\left( S_n - u_n >0 \right) \ \ \mbox{and} \ \
\mu_{0}\left( S_n\leq u_n \right)= \mu_{0}\left( S_n - u_n \leq 0 \right),
\end{equation}
where $S_n$ is defined by \eqref{sma} and $u_n \to E$, when $n$ goes to infinity.
We are interested in Large Deviations for $S_n- u_n$; that is, given an interval $(a,b) \subset \R$, we want to estimate $\mu_j \{ (S_n -u_n) \in (a,b)\}$, for
$j=0,1$. Intervals of the type $(-\infty,0)$ and $(0,\infty)$ are particularly
important.

It is a classical result (see, for instance, \cite{PP}), that
\begin{equation} \label{lol}
\int (\log J_{0} - \log J_{1\,}) \,d \mu_{0} \,>0 \ \ \mbox{and} \ \
\int (\log J_{1} - \log J_{0\,}) d \mu_{1}>0.
\end{equation}
We need to estimate for $j=0,1$.
\begin{equation} \label{lol1}
\mu_{j} ( S_n - u_n \in (a,b))\,=\mathbb{P}_{\mu_{j}}  \left(x\,|\, \frac{1}{n} \sum_{i=0}^{n-1}  \left[\log\left(\frac{J_{0}(\sigma^i (y))}{J_{1\,}(\sigma^i (y))} \right)- u_n \right]\in (a,b)  \right),
\end{equation}

To get the correct large deviation rate, we need first to analyze the expression
\begin{equation}\label{phi}
\phi_n^j(t):=\frac{1}{n}\log\mathbb{E}_{\mu_{j}} \left \{ \exp \left[t\,\, \sum_{i=0}^{n-1}  \left(\log\left(\frac{J_{0}(\sigma^i )}{J_{1\,}(\sigma^i )} \right)- u_n \right) \right]\ \right\},
\end{equation}
for each $n$ and each real value $t$, where $\mathbb{E}_{\mu_{j}}$ denotes the expected value with respect to the probability $\mu_{j}$, for $j=0,1$.
Expression \eqref{phi} is equivalent to
\begin{equation}\label{phi1}
\phi_n^j(t)=\frac{1}{n}\log\left( \int e^{\,  t \sum_{i=1}^{n}(\,\log J_{0}- \log J_{1\,}\,)(\sigma^i (y)) }\mathrm{d} \mu_{j}(x) \right) - t  \,u_n,
\end{equation}
for $j=0,1$.

It is known (see proposition 3.2 in \cite{Kif}, \cite{L4}, theorem 3 in \cite{L2} or \cite{L3}) that
\begin{equation}\label{phi2}
\lim_{n \to \infty}   \frac{1}{n}\log\left( \int e^{\,  t \sum_{i=1}^{n}(\,\log J_{0}- \log J_{1\,}\,)(\sigma^i (y)) }\mathrm{d} \mu_{j}(x) \right)= P( t (\log J_{0} - \log J_{1\,}) + \log J_{j\,}).
\end{equation}

Hence, from the expressions \eqref{phi}, \eqref{phi1} and \eqref{phi2}, one has for $j=0,1$.
\begin{equation}\label{phi3}
\phi^j (t) :=  \lim_{n \to \infty} \phi_n^j(t) =P( t (\log J_{0} - \log J_{1\,}) + \log J_{j\,}) - t \, E,
\end{equation}

Denote by $P_j $ the function
\begin{equation} \label{bourb}
t \to P_j (t):=P( t (\log J_{0} - \log J_{1\,}) + \log J_{j\,}),
\end{equation}
for $j=0,1$. The function $t \to P_j(t)$ is analytic and strictly convex, for $j=0,1$.
Figure \ref{fig:my_label11} shows the graphs of $P_0$ (in solid line) and $P_1$ (in dashed line), for the example in Section \ref{Example}.

One can easily show that, for any $t \in \mathbb{R}$,
\begin{equation} \label{lui}
P_1(t)= P_0(t-1).
\end{equation}

The function
\begin{equation*}
t \to P( t (\log J_{0} - \log J_{1\,}) + \log J_{j\,}) - t E=P_j(t) - t\,E
\end{equation*}
is also convex, for $j=0,1$.

\begin{figure}[h] 	
	\centering
	\includegraphics[scale=0.42]{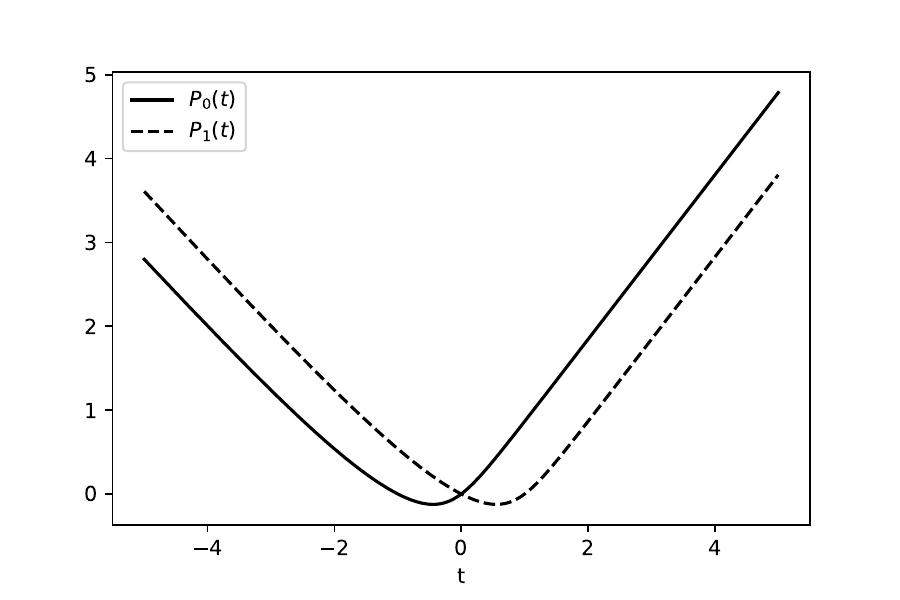}
	\caption{Graphs of $P_0$ (in solid line) and $P_1$ (in dashed line) for the functions defined in \eqref{bourb}. For these plots we consider the data from the example in Section \ref{Example}.}
\label{fig:my_label11}
\end{figure}

Moreover,
\begin{equation}\label{okco}
P_j(0)=P( 0\, (\log J_{0} - \log J_{1\,}) + \log J_{j\,})\,=\,P( \log J_{j\,}) \,=\,0,
\end{equation}
for $j=0,1$. From chapter 4 in \cite{PP}, note that
\begin{equation}\label{derivative}
\frac{d}{dt}P_j(t)|_{t=0} = \frac{d}{dt}\left(\, P( t (\log J_{0} - \log J_{1\,}) + \log J_{j\,}\right)|_{t=0} = \int (\log J_{0} - \log J_{1\,}) d \mu_{j},
\end{equation}
for $j=0,1$.

Then, $\frac{d}{dt} P_0 (t)|_{t=0} >0$ and  $  \frac{d}{dt} P_1 (t) |_{t=0}<0$, if $\mu_1 \neq \mu_0$. Besides,
\begin{equation}\label{derivative1}
\frac{d}{dt}\left(\, P( t (\log J_{0} - \log J_{1\,}) + \log J_{j\,}\right)|_{t}= \int (\log J_{0} - \log J_{1\,}) d \mu_{t}^j ,
\end{equation}
where $\mu_{t}^j$ is the equilibrium probability for $t (\log J_{0} - \log J_{1\,}) + \log J_{j\,}$ (see \cite{PP}).

There exist values $c^{+} > 0> c^{-}$ defined by
\begin{equation*}
c^{-} := \inf_{t \in \mathbb{R}} P_1^{\prime}(t)\ \ \mbox{and} \ \
c^{+} : = \sup_{t \in \mathbb{R}} P_1^{\prime}(t).
\end{equation*}
From expression \eqref{lui} we  also have
\begin{equation*}
c^{-} := \inf_{t \in \mathbb{R}} P_0^{\prime}(t)\ \ \mbox{and} \ \
c^{+} := \sup_{t \in \mathbb{R}} P_0^{\prime}(t).
\end{equation*}

The main interest here is the following: for each value $E$, where $u_n \to E$, one wants
to estimate the asymptotic values of $\mu_{1}\left( S_n - u_n >0 \right) $ and $\mu_{0}\left( S_n - u_n \leq 0 \right) $. The following proposition states the exact values for the deviation function (see \cite{Buck}, \cite{Kif}, \cite{L4} or \cite{L2}).
Before addressing this issue, one observes the following two points:

\begin{itemize}
\item{} From expression \eqref{lui}, it holds
\begin{align}\label{item1}
\nonumber
\frac{d}{d t}P_1(t)|_{t=1} &=\frac{d}{d t}(\, P(t (\log J_{0} - \log J_{1\,}) + \log J_{1\,})|_{t=1}\\
&= \frac{d}{d t} P_0(t)|_{t=0} = \int (\log J_{0} - \log J_{1\,}) \,d \mu_{0} \,>0. \end{align}
\item{} It is also true that
\begin{equation}\label{item2}
\frac{d}{d t} P_1(t)|_{t=0}= \frac{d}{d t} P_0(t)|_{t=-1}.
\end{equation}
\end{itemize}

\begin{rmk} \label{supre}The values of $E$ we shall consider are the ones such that $c^{-} < E < c^{+}$.
\end{rmk}

\begin{prop} For a fixed value $E$, it is true that
\begin{itemize}
\item[(i)] If $E <  \int (\log J_{0} - \log J_{1\,}) d \mu_{1}$,
then
\begin{equation} \label{klm1}
 \lim_{ n \to \infty}\mu_{1}\left( S_n - u_n >0 \right)=1.
\end{equation}
\item[(ii)] If $E>  \int (\log J_{0} - \log J_{1\,}) d \mu_{0}$,
then
\begin{equation} \label{klm2}
 \lim_{ n \to \infty}\mu_{0}\left( S_n - u_n \leq 0 \right)=1.
\end{equation}
\end{itemize}
\end{prop}
\begin{proof} According to  \cite{DZ}, \cite{Ellis} \cite{Kif}, \cite{L4} or \cite{L2},
the {\it large deviation function}  $I_j$, for  $(S_n - u_n)$, $n \in \mathbb{N}$, and for the measure $\mu_j$, is
\begin{align}\label{sup}
\nonumber
I_j(x) &:=\sup_t\left[t x -\phi^j(t) \right]=
\sup_t\left[ t\left( x  + E\, \right)-  P( t (\log J_{0} - \log J_{1\,}) + \log J_{j\,}) \right]\\
&=\sup_t\left[ t\left( x  + E\, \right)-  P_j( t ) \right],
\end{align}
for a fixed value $E$ and $j=0,1$.

The large deviation function is analytic since it is the Legendre transform of an analytic function. That is,
\begin{equation} \label{esqu}
\mu_j \, \{\, x \in \Omega\,|\, (S_n - u_n)(x) \in (a,b)\,\}\sim e^{ -n\,  \inf_{z \in (a,b)} I_j (z)},
\end{equation}
for $j=0,1$. The function $I_j(\cdot)$ is a real analytical one.

Given $E$, take
\begin{equation}\label{vj}
x=v_j:= - E      \,+  \,\left(\int \log J_{0\,}  d \mu_{j}-\int \log J_{1} d \mu_{j}\right).
\end{equation}
By using \eqref{sup}, we get that
$$ I_j(v_j) =\sup_t\left[t( v_j+E )-P_j (t) \right]= $$
$$\sup_t\left[t\,( \int \log J_{0\,}  d \mu_{j}-\int \log J_{1} d \mu_{j})-P_j (t) \right] = 0 - P(0).$$
because by \eqref{derivative}
the supremum is attained at $t=0$. Then, by \eqref{okco} we get
\begin{equation} \label{tri1010}
I_j(v_j)=0, \ \ \mbox{for} \ \ j=0,1.
\end{equation}

Note that from the strict convexity of the pressure it follows that  $I_j$ is zero only in the point $v_j$.  Note that $I$ is strictly monotone in the intervals $(-\infty, v_j)$ and  $(v_j, \infty)$, $j=0,1$.

On the other hand, when  $c^{-}< E < c^{+}$,  from the expressions \eqref{sup} and \eqref{derivative1}, for $x=0$, we get $t^E_j$, where  $t^E_j$, by definition, is the value such that
\begin{align}\label{derivative2}
\nonumber
P_j^{\prime } (t^E_j)& =\frac{d}{dt}(\, P( t (\log J_{0} - \log J_{1\,}) + \log J_{j\,})|_{t^E_j}=E\\
&= \int (\log J_{0} - \log J_{1\,}) d \mu_{t^E_j},
\end{align}
and
\begin{align}\label{tri11}
\nonumber
I_j(0) & =t^E_j E- P( t^E_j (\log J_{0} - \log J_{1\,}) + \log J_{j\,})  = t^E_j E- P_j(t^E_j)\\
\nonumber
&= t^E_j\left( \int (\log J_{0} - \log J_{1}) d \mu_{t^E_j}\right)  - \left[\,t^E_j\left( \int (\log J_{0} - \log J_{1}) d \mu_{t^E_j}\right)\right. \\
& \hphantom{xx} \left.+ \int \log J_j d \mu_{t^E_j}+ h(\mu_{t^E_j}) \right]
= - \left[\, \int \log J_j d \mu_{t^E_j}+ h(\mu_{t^E_j}) \,\right]\,>0,
\end{align}
if $\mu_{t^E_j}\neq \mu_j.$

It follows from \eqref{lui} that $t^E_0= t^E_1-1$ and therefore, $P_0(t^E_0)=P_1(t^E_1)$. From this, follows that
\begin{equation} \label{lulu}
I_1(0)=t^E_1 E- P_1(t^E_1) \ \ \mbox{and} \ \ I_0(0)=  t^E_0 E- P_0(t^E_0) =  I_1(0) - E.
\end{equation}

\vspace{0.2cm}

\noindent \emph{Item $(i)$:} From expression \eqref{esqu}, with $(a,b)=(0,\infty)$,  if $E < \int (\log J_{0} - \log J_{1\,}) d \mu_{1}$, which is equivalent to say that  $v_1>0$, then
\begin{equation*}
 \lim_{ n \to \infty}\mu_{1}\left( S_n - u_n >0 \right)=1,
\end{equation*}
since $ -\inf_{x>0}I_1(x)=0$ (since $v_1 \in (0,\infty)$). Hence, expression \eqref{klm1} is true.

\medskip

Now, we consider the other possibility: $v_1< 0$. As $I_1(v_1)=0$, from \eqref{esqu}, with $(a,b)=(0,\infty)$, we get that $v_1$ does not belong to $(0,\infty)$, and then
\begin{equation} \label{kyr}
\lim_{n\to \infty} \frac{1}{n}  \log ( \mu_{1}\left( S_n - u_n >0 \right)\,)=\lim_{n\to \infty} \frac{1}{n}\log (1- \beta_n)= -\inf_{x>0}I_1(x)= - I_1(0)<0.
\end{equation}

\begin{equation*}
1-\beta_n=\mu_{1}\left( S_n - u_n >0 \right) \sim e^{- n\,\{\inf I_1(x)\,|\, x\geq 0\}  }=e^{ - n \,I_1(0)}\,\to 0.
\end{equation*}
This corresponds to
$$
-E + \int (\log J_{0} - \log J_{1\,}) d \mu_{1} <0.
$$

Note that if $E=0$ then, from the expression \eqref{vj}, we get $v_1=\int (\log J_{0} - \log J_{1\,}) d \mu_{1}<0$, and
\begin{equation} \label{lear}
\mu_{1}\left( S_n>0 \right) \sim e^{ - n \,I_1(0)},
\end{equation}
where $I_1(0)>0$.

\vspace{0.3cm}

\noindent \emph{Item $(ii)$:} On the other hand, if $E >  \int (\log J_{0} - \log J_{1\,}) \,d \mu_{0} \,$, that is, $v_0<0$, then {\bf as $I_0 (v_0)=0$,}
\begin{equation} \label{klm3}
 \lim_{ n \to \infty}\mu_{0}\left( S_n - u_n \leq 0 \right)=1,
\end{equation}
since $-\, \inf_{x<0 }I_0(x)=0$ (since $v_0\in (-\infty,0)$). Expression \eqref{klm3} is true by \eqref{esqu}, with $(a,b)= (-\infty, 0)$.

\medskip

If $v_0>0$, as $I_0(v_0)=0$, from \eqref{esqu}
\begin{equation*}
\lim_{ n \to \infty}\frac{1}{n}  \log (\mu_{0}\left( S_n - u_n \leq 0 \right)\,)= -\inf_{x<0}I_0(x)= - I_0(0)<0.
\end{equation*}

Then, we obtain
\begin{equation*}
\mu_{0}\left( S_n - u_n \leq 0 \right) \sim e^{- n\,\{\inf I_0(x)\,|\, x\leq 0\}} =e^{ - n \,I_0(0)} \to 0.
\end{equation*}
\end{proof}
\qed

\section{Dynamical  Hypotheses Test} \label{NP}
\renewcommand{\theequation}{\thesection.\arabic{equation}}
\setcounter{equation}{0}

In Section \ref{HypoTest} we introduced sequences $\tilde{u}_n$, $n \in \mathbb{N}$,
their corresponding limits $E$ and the rejection regions indexed by $n$, which determine, the acceptance or rejection regions for the test, for each $n \in \mathbb{N}$. Each one of these limit possible values $E$ will determine a certain test and this is the class of tests we are interested here. Each of these tests will determine a certain decay rate of decreasing the probability of wrong decisions.

In this section, we shall deal with a class of dynamical hypothesis test.
Its optimality property will be shown in Subsection \ref{Optimal}.

From the ergodicity of $\mu_{0}$, we have
\begin{equation} \label{mest}
S_n \stackrel{a.s\, (\mu_0)}{\longrightarrow}\int \log\left(\frac{J_{0}}{J_{1\,}} \right)\mathrm{d} \mu_{0}=  \int \log J_0 \,\,d \mu_{0} \, \,- \int \log\,J_1\,\mathrm{d}\mu_{0},
\end{equation}
\noindent whenever hypothesis $H_0$ is true (the  time series $S_n(y)$  are obtained from the randomness of the probability measure $\mu_{0}$).

For the $NP$ test, we choose  a sequence  $u_n$, such that
\begin{equation} \label{kli}
\lim_{n \to \infty} u_n = E^{NP}:= \int \log J_0 \,d \mu_{0} \,-   \int \log\,J_1\,\mathrm{d}\mu_{0}.
\end{equation}

This defines a certain rejection region

\begin{equation} \label{yoko}
\mathcal{R}_n^{NP}:=\{x\in \Omega \,|\, S_n< u_n\}, \ \ n \in \mathbb{N}.
\end{equation}

Each possible test, an alternative $A$ for the $NP$ test, is associated with a certain monotonous sequence $\tilde{u}_n$ which converges to a determined value $E$. This defines a certain rejection region

\begin{equation} \label{yokob}
\mathcal{R}_n^A:=\{x\in \Omega \,|\, S_n< \tilde{u}_n\}, \ \ n \in \mathbb{N}.
\end{equation}
Later we shall compare the asymptotic rate error for different choices of limits $E$, for the sequences $\tilde{u}_n$, $n \in \mathbb{N}$. This requires to estimate the asymptotic in $n$ of the probability of the rejection regions $\mathcal{R}_n^A$.

For analyzing the asymptotic values of $\mu_{0}(S_n\leq \tilde{u}_n)$, associated to the \emph{type I error}, when $n$ goes to infinity, it is necessary to assume that
\begin{equation} \label{yoko234}
\tilde{u}_n\longrightarrow   E>E^{NP}= \int \log J_0 \,d \mu_{0} \,-   \int \log\,J_1\,\mathrm{d}\mu_{0} >0.
\end{equation}

This is so because, from \eqref{mest},
$$
S_n \stackrel{a.s\, (\mu_0)}{\longrightarrow} E^{NP}=  \int \log J_0 \,\,d \mu_{0} \, \,- \int \log\,J_1\,\mathrm{d}\mu_{0},
$$
in the analysis of type I error (otherwise, $\mu_{0}(S_n<\tilde{u}_n)$ would converge to  $1$).

\medskip

Later, in Subsection \ref{Optimal}, we will show that  $E^{NP}$ is optimal when compared with the other tests (associated with the other values of $E\neq E^{NP}$).  This optimality will be considered under  certain conditions on the sequences $u_n$ and $\tilde{u}_n$:  for each $n \in \mathbb{N}$, the alternative hypotheses test $A$ has a smaller or equal false alarm probability than the $NP$ test.

The \emph{probability of misspecification} is given by $\prob(\mbox{Decide} H_0\, |\, H_1 \, \mbox{is true})=1-\beta$. We are interested in minimizing the misspecification probability.
We shall consider large samples of size $n$ and shall apply classical results on large deviations theory. In the Neyman-Pearson  Lemma we shall consider the \emph{rejection region} given by \eqref{yoko} under the assumption that \eqref{kli} is true. Among other things, for the $NP$  test, we want to estimate the asymptotic misspecification probability $1-\beta_n,$
$n\in \mathbb{N}$. From an experimental point of view, the  relevance of the information of a  time series of size $n$ obtained from the randomness of the  measure $\mu_1$  is related to this asymptotic value.

As mentioned before
\begin{equation*}
1-\beta_n=\mu_{1}\left( S_n\geq  u_n \right)= \mu_{1}\left( S_n - u_n \geq 0 \right).
\end{equation*}

From expression \eqref{vj}, when $E=E^{NP}= \int \log J_0 \,d \mu_{0} \,-   \int \log\,J_1\,\mathrm{d}\mu_{0}$ (see expression \eqref{kli}),
we get
\begin{align}\label{cine}
\nonumber
v_1 = &- E^{NP} \,+  \,\left(\int \log J_{0\,}  d \mu_{1}-\int \log J_{1} d \mu_{1}\right)
= - \left( \int \log J_0 \,d \mu_{0} \,-   \int \log\,J_1\,\mathrm{d}\mu_{0}  \right ) \\
&+  \,\left(\int \log J_{0\,}  d \mu_{1}-\int \log J_{1} d \mu_{1}\right)<0.
\end{align}

Then, $v_1<0$ and, from expressions \eqref{tri1010}  and  \eqref{kyr}
\begin{equation} \label{tri101}
I_1(v_1)=0\,\,\text{and} \,\,\,- I_1(0) <0.
\end{equation}

When $ E^{NP}= \int \log J_0 \,d \mu_{0} \,-   \int \log\,J_1\,\mathrm{d}\mu_{0}$, we get from \eqref{sup} and \eqref{item1}, that $t^{E^{NP}}_1=1$. From expressions \eqref{kyr}, \eqref{sup},  \eqref{esqu}, \eqref{tri1010}  and \eqref{tri11} we get
\begin{align}\label{limit2}
\nonumber
\lim_{n\to \infty} \frac{1}{n}  \log ( \mu_{1}\left( S_n - u_n >0 \right)\,)
&= - I_1(0)\\
& = -E^{NP} = - \left( \int \log J_0 \,\,d \mu_{0}
 - \int \log \,J_1\,  \,\,d \mu_{0} \, \,\right)<0.
\end{align}
Expression \eqref{limit2} shows that our class of dynamical hypotheses test performs well,
since the probability of misspecification, $1-\beta_n$ goes exponentially fast to $0$.
The above expression can be seen as a version of Stein's Lemma (see, for instance, \cite{Cover}).

Note that if $-E<0$, then a similar inequality as \eqref{cine} is true, and from this  follows that
$- I_1(0) <0.$ Therefore,
\begin{equation}\label{limit21}
\lim_{n\to \infty} \frac{1}{n}  \log ( \mu_{1}\left( S_n - \tilde{u}_n >0 \right)\,)
= - I_1(0)  <0.
\end{equation}

In this case, it is also true that the misspecification probability goes exponentially fast to $0$.

\subsection{The  Optimality Property}\label{Optimal}

{\bf In this section, we shall prove the optimality property associated with the value $E^{NP}=\int \log J_0 \,d \mu_{0} \,-   \int \log\,J_1\,\mathrm{d}\mu_{0}$}. We are interested in   maximizing the test power (which is equivalent to minimizing the type II error probability $1- \beta_n$, $n \in \mathbb{N}$).
We will show that  there is no other alternative hypothesis $A$ ($E\neq e^{NP}$ ) that provides a smaller mean value error when the size $n$ of the time series  goes to infinity.

We assume that the monotonous sequence  $u_n$, $n \in \mathbb{N}$, in the $NP$ tests is such that $u_n \to E^{NP}$. For each possible value $E$ we consider a monotonous sequence $\tilde{u}_n$, such that  $\tilde{u}_n\to E$.  Without loss of generality, we assume that any monotonous sequence $\tilde{u}_n$, $n \in \mathbb{N}$, can represent
a given value $E$.

For the alternative test $A$ the limit value $E$ is the most important issue and not the specific values $\tilde{u}_n$. Note that given the sequence $\mu_{0}\left(S_n<u_n \right)$ there exists a monotonous sequence
$\tilde{u}_n$, such that,
\begin{equation} \label{relation22}
\mu_{0}\left(S_n<\tilde{u}_n \right)=\mu_{0}\left(S_n<u_n \right).
\end{equation}
This  shows that the  future  condition \eqref{relation} does not concern an empty set of cases.

Remember that the \emph{type II error} at time $n$ is denoted by $\mu_{1}\left( S_n\geq  \tilde{u}_n \right) =1-\tilde{\beta}_n=1 -\mu_1(\mathcal{ R}_n^A)$.

In the next theorem, we show that $E=E^{NP}=\int \log J_0 \,d \mu_{0} \,-   \int \log\,J_1\,\mathrm{d}\mu_{0}$ is the best choice for $E\neq E^{NP}$, in terms of getting
$$
1-\tilde{\beta}_n=1- \mu_1(\mathcal{ R}_n^A)>1-\mu_1( \mathcal{R}_{n}^{NP}),
$$
for large $n$ and for any alternative test $A$.  This will be  proved under
the condition \eqref{relation} for $u_n$ and $\tilde{u}_n$, which means that
the alternative hypotheses test $A$ has smaller or equal false alarm probability
than the $NP$ test.

\begin{thm}
\emph{ Consider the test where}
\begin{equation*}
u_n \to E^{NP}=  \int \log J_0 \,d \mu_{0} \,-   \int \log\,J_1\,\mathrm{d}\mu_{0}.
\end{equation*}
When comparing  this test,  with the  tests for other possible values of  $E\neq E^{NP}$, the one associated to $E^{NP}$ minimizes the type II error.

The best decay rate for minimizing the wrong decisions probabilities of the $E^{NP}$ test will be of order $\exp\{-\,n\, (\,\int \log J_0 \,d \mu_{0} \,-   \int \log\,J_1\,\mathrm{d}\mu_{0}\,)\}$.

The above is equivalent to say that
$$
\lim_{n \to \infty} \frac{1}{n} \log (\,\mu_{1} ( S_n> \tilde{u}_n)\,) <  \lim_{n \to \infty} \frac{1}{n} \log (\,\mu_{1} ( S_n> u_n)\,)$$
$$ =\,-\left(\int \log J_0 \,d \mu_{0} \,-   \int \log \,J_1\,\,d\mu_0\,\right)\, =\,-  E^{NP},$$
under the hypotheses
$$\mu_{0}\left(S_n<\tilde{u}_n \right)\leq \mu_{0}\left(S_n<u_n \right).$$
\end{thm}
\proof  For the sequence $u_n$ (which converges to $E^{NP}$) we set
\begin{equation*}
\mu_{0}\left(S_n<u_n \right):=\alpha_n, \,n \in \mathbb{N}
\end{equation*}
which  describes \emph{the false alarm} in the $NP$ test.

We denote any other alternative hypotheses test by $A$.
To each different alternative hypotheses test $A$ corresponds  a different choice of $E$, such that, $\tilde{u}_n\to E=E^A$.

We are interested in optimizing the value of the probability of announcing $H_1$ when $H_1$ is true, that is, in optimizing the value $\tilde{\beta}_n=\mu_1 (S_n - \tilde{u}_n<0)$, $n \in \mathbb{N}$. The \emph{rejection region} for the $NP$ hypotheses test, at time $n$, is denoted by $\mathcal{R}_{n}^{NP}$ and defined as
\begin{equation*}
\mathcal{R}_{n}^{NP}:= \left\{x \in \Omega|
S_n(x)  <u_n\right\}.
\end{equation*}
\noindent Similarly, the \emph{rejection region} for another alternative hypotheses test $A$, at time $n$, will be denoted by $\mathcal{R}_{n}^{A}$, where
\begin{equation*}
\mathcal{R}_{n}^{A}:= \left\{x \in \Omega|
S_n(x) <\tilde{u}_n\right\},
\end{equation*}
\noindent for another sequence $\tilde{u}_n>0$.

Assume that $\tilde{u}_n$ is such that

\begin{equation} \label{relation}
\mu_{0}\left(S_n<\tilde{u}_n \right)\leq\alpha_n=\mu_{0}\left(S_n<u_n \right),
\end{equation}
for all $n$. This means
\begin{equation}\label{relation2}
-\mu_{0}\left(S_n<u_n \right) + \mu_{0}\left(S_n<\tilde{u}_n \right)\leq 0.
\end{equation}
Note that $\tilde{u}_n \leq u_n$, for all $n$. The inequality \eqref{relation2} means we are assuming that the alternative hypotheses test $A$ has a smaller or equal false alarm probability. In other words, we do not consider tests with larger size test
than the one obtained in the  $NP$ case. Each choice of the sequence $\tilde{u}_n$, $n \in \N$, will be understood as an alternative possible hypothesis. For each alternative hypotheses test $A$, we set the associated value
\begin{equation*}
\tilde{\beta}_n \, :=\, \mu_{1} ( S_n< \tilde{u}_n),
\end{equation*}
for any $n \in \mathbb{N}$.

Remember that  $\tilde{\beta}_n = \mu_1(\mathcal{ R}_n^A)$ and $\beta_n =\mu_1( \mathcal{R}_{n}^{NP})$. The optimality property means to compare the value $1- \beta_n$ of the  $NP$ test with the value $1-\tilde{\beta}_n$ of  another alternative test $A$. This will be   based on the  values
$\mathcal{ R}_n^A,$ $n \in \mathbb{N},$ using a sequence $\tilde{u}_n$ satisfying \eqref{relation} and, assuming  that
\begin{equation}  \label{G}
E:= \lim_{n\to \infty}   \tilde{u}_n < \lim_{n\to \infty}   u_n= \int \log J_0 \,d \mu_{0} \,-   \int \log\,J_1\,\mathrm{d}\mu_{0}=E^{NP},
\end{equation}
due to \eqref{yoko234}.

We shall assume $E>0$. We point out that,  for the alternative hypotheses test $A$, the associated value $\tilde{\beta}_n$ satisfies
\begin{equation*}
\mu_1 (S_n - \tilde{u}_n <0)\,= \mu_1(\mathcal{ R}_n^A)= \tilde{\beta}_n.
\end{equation*}

We want to show that
\begin{equation*}
\tilde{\beta}_n  <\beta_n =\mu_1( \mathcal{R}_{n}^{NP} )=\mu_1 (S_n - u_n <0),
\end{equation*}
\noindent for large $n$. More precisely, we want to show that
\begin{equation} \label{orty}
\lim_{n \to \infty} \frac{1-\beta_n}{1-\tilde{\beta}_n}\,=0.
\end{equation}

The expression \eqref{orty} will guarantee that the $NP$ test is exponentially better than any other alternative test $A$.

It is known that
\begin{equation*} \label{zz}
1-\beta_n \sim e^{-n\,  I_1(0) }= e^{-n (\int \log J_0 d \mu_0 - \int \log \,J_1\, d  \mu_0)}.
\end{equation*}
Now we will show that, for large $n$,
\begin{equation} \label{rr}
\mu_{1} (\mathcal{R}_{n}^{NP})= \beta_n=\mu_{1} ( S_n< u_n) \geq  \tilde{\beta}_n = \mu_{1} ( S_n< \tilde{u}_n)  = \mu_{1}
(\mathcal{ R}_n^A),
\end{equation}
or, equivalently, that
\begin{equation} \label{rrr}
1-\beta_n=\mu_{1} ( S_n\geq  u_n) \leq  1-\tilde{\beta}_n \, =\, \mu_{1} ( S_n> \tilde{u}_n).
\end{equation}

Note that
\begin{equation} \label{kwr}
\mu_{1} ( S_n> \tilde{u}_n)  = \mu_{1} ( S_n - u_n> \tilde{u}_n-u_n).
\end{equation}

As $u_n \to (\int \log J_0 \,d \mu_{0} \,-   \int \log\,J_1\,\mathrm{d}\mu_{0})$ and $\tilde{u}_n \to E$, from expressions \eqref{kwr}, \eqref{sup}, \eqref{esqu} and \eqref{tri11}, we get the following limit
\begin{equation*}
\lim_{n \to \infty} \frac{1}{n} \log (\,\mu_{1} ( S_n> \tilde{u}_n)\,) =
 - \, \inf \left \{ I_1(x)\,|\, x \geq  E - \left(\int \log J_0 \,d \mu_{0} \,-   \int \log\,J_1\,\mathrm{d}\mu_{0} \right)\right \}.
\end{equation*}

\medskip

We set
\begin{equation} \label{pou}
G_1 : = E - \left(\int \log J_0 \,d \mu_{0} \,-   \int \log \,J_1\,\mathrm{d}\mu_{0} \right),
\end{equation}
and, from  expression \eqref{G} we get that $G_1<0$.

\begin{figure}[!htb]
	\centering
	{\includegraphics[scale=0.27]{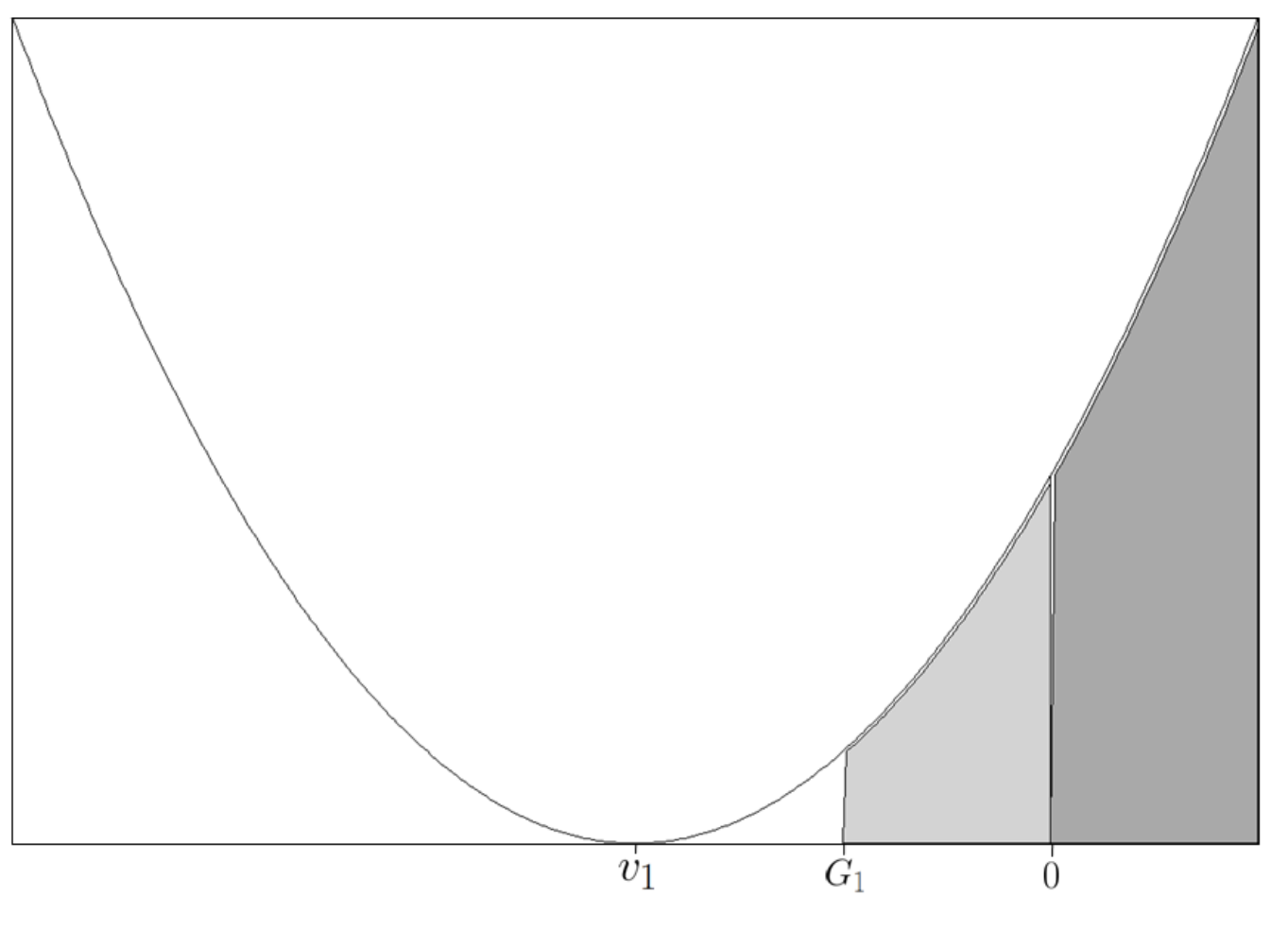}}
\caption{The large deviation rate function $I_1(\cdot)$ at points $v_1$,
$G_1$ and zero, where $G_1 = E - \left(\int \log J_0 \,d \mu_{0} \,-   \int \log \,J_1\,\,d\mu_0\,\right)$.}
\label{Figure_4.1}
\end{figure}

Observe that $v_1< G_1$. Indeed, as $E>0$, from the expression \eqref{pou}, one has
\begin{align}
\nonumber
v_1 & =    \,\left(\int \log J_{0\,}  d \mu_{1}-\int \log J_{1} d \mu_{1}\right)  -\left ( \int \log J_0 \,d \mu_{0} \,-   \int \log\,J_1\,\mathrm{d}\mu_{0}   \right)\\
\nonumber
& < E - \left(\int \log J_0 \,d \mu_{0} \,-   \int \log \,J_1\,\mathrm{d}\mu_{0} \right).
\end{align}

Note that from the strict convexity of the pressure it follows that its Legendre transform $I_1$ is strictly monotone on the interval $(v_1,0)$.

Then, from the expression \eqref{pou}, we obtain
\begin{align*}
\nonumber
I_1 (G_1) &=  I_1\left( G - \left(\int \log J_0 \,d \mu_{0} \,-   \int \log\,J_1\,\mathrm{d}\mu_{0} \right)\right)<
\int \log J_0 \,d \mu_{0} \,-   \int \log \,J_1\,\,d\mu_0\\
\nonumber
& = I_1(0).
\end{align*}

From the above we get that
\begin{equation} \label{expp}\lim_{n \to \infty} \frac{1}{n} \log (\,\mu_{1} ( S_n> \tilde{u}_n)\,) =
-I_1(G_1)  >- I_1(0) .
\end{equation}

Therefore,
\begin{align*}
\lim_{n \to \infty} \frac{1}{n} \log (\,\mu_{1} ( S_n> \tilde{u}_n)\,)& =
-I_1(G_1) >- I_1(0) \\
& = -\left(\int \log J_0 \,d \mu_{0} \,-   \int \log \,J_1\,\,d\mu_0\,\right) \\
& =  \lim_{n \to \infty} \frac{1}{n} \log (\,\mu_{1} ( S_n> u_n)\,).
\end{align*}

This proves that the expressions \eqref{rrr} and \eqref{orty} are true.
\qed

Figure \ref{Figure_4.1}  shows the large deviation rate function $I_1(\cdot)$ at points $v_1$, $G_1$ and zero, where the point $G_1$ is given by \eqref{pou}.
\section{The Min-Max Hypotheses Test} \label{MinMax}
\renewcommand{\theequation}{\thesection.\arabic{equation}}
\setcounter{equation}{0}

In this section, we shall present the Min-Max Hypotheses test in the dynamical
sense. Once more, for $S_n$, $n \in \mathbb{N}$, given by \eqref{sma},
$\mu_1 \{x\,|\, S_n>  u_n\}$ is associated to a wrong decision by
announcing $H_0$ when $H_1$ is true while $\mu_0\{x\,|\, S_n\leq  u_n\}$
is associated to a wrong decision by announcing $H_1$ when $H_0$ is true.
For each value $E$ we consider a sequence $u_n$, such that,
\eqref{yuo} holds. In the same way as before, the limit value $E$ is more
important than the specific values $u_n$.

In this section we shall consider large deviation properties
for both $\mu_1 \{x\,|\, S_n>  u_n\}$ and $\mu_0 \{x\,|\, S_n<  u_n\}$.
For the Min-Max hypotheses test, we need loss functions for a false alarm. This
is a classical ingredient in Hypotheses Tests (see \cite{Cover} or \cite{Roh}).

Here we consider the case when the loss functions for false alarm for $H_0$ and $H_1$ are constants, respectively, given by $y_0$ and $y_1$. The main question here is once
again what is the best value for $E$? The idea behind the use of loss functions is wrong decisions can have a cost. We are interested in finding some optimality property in this setting. From the large deviation properties for this setting, for each choice of limit value $E$ we shall obtain $C_1(E)=C_1 > 0$ and $C_0(E)=C_0> 0$, such that,
\begin{equation*}
\mu_1 \{x\,|\, S_n>  u_n\}\sim \,e^{-C_1\,   n}\ \ \mbox{and} \ \ \mu_0 \{x\,|\, S_n\leq   u_n\}\sim \, e^{-C_0\, n}.
\end{equation*}

In the Min-Max hypotheses test, we have to compare the  asymptotic values of the maximum of
\begin{equation} \label{mmi} y_1\,\mu_1 \{x\,|\, S_n>  u_n\}\sim y_1\,e^{-C_1\,   n} \, \,\,\text{and} \,\,
y_0\, \mu_0 \{x\,|\, S_n\leq   u_n\}\sim y_0\, e^{-C_0\, n},
\end{equation}
that shall take into account the loss functions $y_1$ and $y_0$, for each value $E$.
Finally, we shall consider the minimum $\tilde{E}$ among all possible values of $E$,
that is, the minimum of the function
$E \to \max\{C_0(E), C_1(E)\}$. This means that in the Min-Max hypotheses test we are interested in minimizing the maximal cost of wrong decisions by taking either $H_0$ or $H_1$.
Our main result is:

\begin{thm} \label{eer} In the Min-Max hypotheses test, the best choice of $E$ will be $E=0$. Moreover, the best decay rate for minimizing the probability of wrong decisions is
given by $e^{n \,r}$, where $r<0$ is the minimum value of the pressure function $P_1$. The value
$r$ is determined by expression \eqref{derivative123}.\end{thm}

Given an interval $(a,b) \subset \mathbb{R}$ we will be interested simultaneously in Large Deviation properties for $S_n- u_n$, where $u_n \to E$. Then, we have to estimate both $\mu_1 \{ (S_n -u_n) \in (a,b)\}$ and   $\mu_0 \{ (S_n -u_n) \in (a,b)\}$.  This problem was addressed in Section \ref{LDP}. The main properties we shall need in the Min-Max hypotheses test are related to the deviation functions $I_1$ and $I_0$ values.
From Section \ref{LDP}, for each value $E$, we obtain the corresponding values $t^E_1$ and $t^E_0$, such that
\begin{equation*}
\frac{d}{dt} P_1  (t^E_1)=E \ \ \mbox{and} \ \   \frac{d}{dt} P_0  (t^E_0)=E.
\end{equation*}
From the expression \eqref{lui}, we obtain $t^E_0=t^E_1 -1$. And, from expression
\eqref{tri11}, we get
\begin{equation*}
I_j(0) =   t^E_j \,E- P_j(t^E_j),
\end{equation*}
for $j=0,1$.
Recall, from expression \eqref{lulu}, that $I_0(0) = I_1(0)- E$.

Denote by $c^+>0$, the limit of the derivative
\begin{equation*}
\frac{d}{dt} P_1(t)|_{t} =\frac{d}{dt}(\, P(t (\log J_{0} - \log J_{1\,}) + \log J_{1\,})|_{t},
\end{equation*}
when $t$ goes to infinity. The value $c^+$ is the maximal value of the ergodic optimization
for the potential $\log J_{0} - \log J_{1\,}$ (see \cite{BLL} and also Section \ref{Example}). When $E \to c^{+}$, we get that
\begin{equation*}
 P_1(t^E_1) -E= P_1(t^E_1) -   \frac{d}{dt}  P_1(t)|_{t=t^E_1}\to  \infty,
\end{equation*}
since $t^E_1$ goes to infinity. On the other hand, denote by $c^{-}<0$,
the limit of the derivative
\begin{equation*}
\frac{d}{dt} P_0(t)|_{t} =\frac{d}{dt}(\, P(t (\log J_{0} - \log J_{1\,}) + \log J_{0\,})|_{t},
\end{equation*}
when $t\to -\infty$. The value $c^-$ is the maximal value of the ergodic optimization
for the potential $\log J_{1} - \log J_{0\,}$ (see \cite{BLL} and also Section \ref{Example}).
When $ E \to c^{-}$, we get that
\begin{equation*}
P_0(t^E_0) -E= P_0(t^E_0) -  \frac{d}{dt} P_0(t)|_{t=t^E_0} \to \infty,
\end{equation*}
since $t^E_0 \to - \infty$.

\begin{rmk} \label{rema}
Observe that both $I_0(0)\geq 0$ and $I_1(0)\geq 0$ {\bf depend on} $E$.
Only the values of $E$, such that, $I_1(0)=t^E_1 \,E - P_1 (t^E_1)\geq 0$  and $I_0(0)=t^E_0 \,E - P_0 (t^E_0)\geq 0$, are relevant for the Min-Max hypotheses test analysis. Indeed, if one of the two options do not happen, then
\begin{equation*}
y_1\,\mu_1 \{x\,|\, S_n>  u_n\}\sim y_1\,e^{-C_1\,   n} \ \ \mbox{or} \ \
y_0\, \mu_0 \{x\,|\, S_n\leq   u_n\}\sim y_0\, e^{-C_0\, n}
\end{equation*}
will be large and this value of $E$ should be discarded in the search for the optimal $\tilde{E}$.
\end{rmk}

Recall that from expressions \eqref{klm1} and \eqref{klm2}, if $ v_0<0$, then $I_0(0)=0$ and,  if $ v_1>0$, then $I_1(0)=0$.

If $C_0>C_1$, then the dominant part of the maximum of \eqref{mmi} is $ y_1\,e^{- C_1 n}$. On the other hand, if $C_1>C_0$, then the dominant part of the maximum of the
same expression \eqref{mmi} is $y_0\, e^{- C_0 n}$. The specific values of $y_0$ and $y_1$ are irrelevant for this test and we just have to look for the {\bf minimum value} $\tilde{E}$ of the function
\begin{equation*}
E \to r(E):= \max\{\, \inf \{ I_0 (x)\,|\, x\leq 0\} , \inf \{ I_1 (x)\,|\, x\geq 0\}\, \},
\end{equation*}
but only for values $E$ such that both $I_0 (0)> 0$ and $I_1(0) >0$.

From the expression \eqref{lulu}, we have that $ I_0(0)=  t^E_0 E- P_0(t^E_0) =  I_1(0) - E$. Then, we just have to find  the minimum value $\tilde{E}$ of the function
\begin{equation} \label{casea}  E \to \inf \{  I_1 (0) , I_0 (0)   \} =\inf \{  I_1 (0) , I_1 (0) - E  \} ,
\end{equation}
for values $E$, such that both $I_0(0)>0$ and $I_1(0)>0$.

From Section \ref{LDP}, we have that $\mu_1 \{x\,|\, S_n>  u_n\}\sim e^{- n\, \inf \{ I_1 (x)\,|\, x\geq 0\} }$, which depends on each value of $E$ through the limit \eqref{yuo}. These values $\mu_1 \{x\,|\, S_n>  u_n\} $ will be maximum when $\inf \{ I_1 (x)\,|\, x\geq 0\}$ is minimum. In the same way, according to Section \ref{LDP},
for each value of $E$, we have that $\mu_0 \{x\,|\, S_n>  u_n\}\sim e^{- n\, \inf \{ I_0 (x)\,|\, x\leq 0\} }$. These values   $\mu_0 \{x\,|\, S_n\leq   u_n\}$      will be maximum when $\inf \{ I_0 (x)\,|\, x\leq 0\}$ is minimum.

In the search for the optimal $\tilde{E}$, we consider several different cases according to the position of $E$ in the set $(c^{-},c^{+})$.

\begin{itemize}
\item {\bf Case 1:} $c^{-}<E< \int (\log J_{0} - \log J_{1\,})\, d \mu_1 <0<  \int (\log J_{0} - \log J_{1\,})\, d \mu_0< c^{+}$.
In this situation, $\inf \{ I_0 (x)\,|\, x\leq 0\}= E- P_0(t^E_0) $ and  $ \inf \{ I_1 (x)\,|\, x\geq 0\}=0$. Hence, $v_1>0$ and $v_0>0$.
Therefore, such values of $E$ should be discarded according to Remark \ref{rema}.

\item {\bf Case 2:} $ c^{+}<\int (\log J_{0} - \log J_{1\,})\, d \mu_1 <0<  \int (\log J_{0} - \log J_{1\,})\, d \mu_0<E<c^{+}$.
In this situation,   $\inf \{ I_0 (x)\,|\, x\leq 0\}= 0 $ and
$ \inf \{ I_1 (x)\,|\, x\geq 0\}= E - P_1(t^E_1)$.
Hence, $v_1<0$ and $v_0<0$. Therefore, such values of $E$ should be discarded according to Remark \ref{rema}.

\item {\bf Case 3:} $\int (\log J_{0} - \log J_{1\,})\, d \mu_1 \leq E \leq  \int (\log J_{0} - \log J_{1\,})\, d \mu_0$. As $r$ is a continuous function it follows from the above that there exists a minimum $\tilde{E}$ for the function described by \eqref{casea} restricted to this interval of values of $E$. This corresponds to $v_1<0$ and $v_0>0$.
\end{itemize}

Observe that, in {\bf Case 3}, $t^E_1 $ range in an increasing monotonous way from $0$ to $1$. From Section \ref{LDP} we obtain $\inf \{ I_0 (x)\,|\, x\leq 0\}= t^E_0\, E- P_0(t^E_0) $ and $ \inf \{ I_1 (x)\,|\, x\geq 0\}=t^E_1\, E- P_1(t^E_1)$.
When $\int (\log J_{0} - \log J_{1\,})\, d \mu_1\leq E<0$, from \eqref{casea}
we obtain
\begin{equation} \label{cnn7}
r(E)=  t^E_1\,E \,- P_1(t^E_1)- E =  (t^E_1-1)\,E  - P_1(t^E_1),
\end{equation}
and, when $\int (\log J_{0} - \log J_{1\,})\, d \mu_0\geq  E>0$, from \eqref{casea}
we obtain
\begin{equation} \label{cnn8}
r(E)=  t^E_1\,E\, -\, P_1(t^E_1).
\end{equation}

We shall analyze the following two functions: for $\int (\log J_{0} - \log J_{1\,})\, d \mu_1 \leq E \leq  \int (\log J_{0} - \log J_{1\,})\, d \mu_0$
\begin{equation} \label{sil2}
E \to P_1(t^E_1) -t^E_1\,E +E =  P_1(t^E_1) -(t^E_1\,-1)   P_1^{\prime} (t^E_1)
\end{equation}
and
\begin{equation} \label{sil1}
E \to P_1(t^E_1) -t^E_1\,E=  P_1(t^E_1) -t^E_1\,   P_1^{\prime} (t^E_1).
\end{equation}
Observe that \eqref{sil1} is a monotonous decreasing function. Indeed,
$$
\frac{d}{d E} [\,P_1(t^E_1) -t^E_1\,   P_1^{\prime} (t^E_1)\,]= - \,(  t^E_1)\, (  t^E_1)^{\prime}\, P_1^{\prime \prime} ( t^E_1  )<0,
$$
since $(  t^E_1)^{\prime}>0$, $t^E_1>0$ and the result follows from the convexity.
On the other hand, expression \eqref{sil2} is a monotonous increasing function
since
$$
\frac{d}{d E} [\,P_1(t^E_1) -t^E_1\,   P_1^{\prime} (t^E_1)+ P_1^{\prime} (t^E_1) \,]=  \,(1-  t^E_1)\, (  t^E_1)^{\prime}\, P_1^{\prime \prime} ( t^E_1  )>0,
$$
since $ (1-  t^E_1)$, $(t^E_1)^{\prime}>0$ and the result follows from the convexity.

When $E= \int (\log J_{0} - \log J_{1\,})\, d \mu_1$, which is a negative value, we have
$t^E_1=0$ and $P_1(t^E_1)- t^E_1  E +E=P_1(0)- 0  E +E= E$. Hence, expression \eqref{sil2} is equal to $\int (\log J_{0} - \log J_{1\,})\, d \mu_1<0$.
On the other hand, when $E= \int (\log J_{0} - \log J_{1\,})\, d \mu_0$, which is
a positive value, we have $t^E_1=1$ and $P_1(t^E_1)=0$. Hence,  $ P_1(t^E_1)-E+E=-E+E=0$.
This describes the values of the function given by \eqref{sil2} on the interval $\int (\log J_{0} - \log J_{1\,})\, d \mu_1 <E <  \int (\log J_{0} - \log J_{1\,})\, d \mu_0$. Note that when $E=0$ the two functions \eqref{sil1} and \eqref{sil2} coincide.

Now, we shall analyze the function \eqref{sil1}.  At the point $E= \int (\log J_{0} - \log J_{1\,})\, d \mu_1<0$, we get that $t_1^E= 0$, and \eqref{sil1} is equal to $0$. Moreover, when $E= \int (\log J_{0} - \log J_{1\,})\, d \mu_0>0$, then, $t^E_1=1,$ and the function \eqref{sil1} attains the value
$$
-  \int (\log J_{0} - \log J_{1\,})\, d \mu_0 <0.
$$
Therefore, there exists a unique point $E$ where the two functions \eqref{sil1} and \eqref{sil2} are equal. This is the value $\tilde{E}=0$, and then, $t_0^1$ is such that $P_1^{\prime} (t_0^1)=0$. Therefore,
$$
r(\tilde{E})=  r(0)=P_1(t_E^1) -t_0^1\,E\, =  P_1(t_0^1),
$$
which is the {\bf minimum value of the function $P_1$}.

It follows from the above  and from expression \eqref{casea} that the minimum of $r(E)$ on the interval  $\int (\log J_{0} - \log J_{1\,})\, d \mu_1 <E <  \int (\log J_{0} - \log J_{1\,})\, d \mu_0$  is attained when $\tilde{E}=0$.
The point $\tilde{E} $ satisfies  $P_1 ^{\prime} (t_{\tilde{E}}^1) =
P_1^{\prime} (t_0^1)=0$.

In this case, the Min-Max solution for $\tilde{E}$ satisfies
\begin{equation} \label{mmi17}
\min \max \{\, y_1\,\mu_1 \{x\,|\, S_n>  u_n\}, \,
y_0\, \mu_0 \{x\,|\, S_n\leq   u_n\}\,\,\}\sim \, e^{P_1(t_0^1)\, n},
\end{equation}
giving us the best rate.

From \eqref{derivative1} we get that
\begin{equation}\label{derivative12}
\frac{d}{dt}P_1 (t)|_{t_0^1} =\frac{d}{dt}\left(\, P( t (\log J_{0} - \log J_{1\,}) + \log J_{1\,}\right)|_{t_0^1}= \int (\log J_{0} - \log J_{1\,}) d \mu_{t_0^1}^1=0,
\end{equation}
where $\mu_{t_0^1}^1$ is the equilibrium probability for $t_0^1 (\log J_{0} - \log J_{1\,}) + \log J_{1\,}$.

Therefore,
\begin{equation}\label{derivative123}r= P_1( t_0^1) = h(\mu_{t_0^1}^1) + t_0^1 \int ( \log J_0 - \log J_1) d \mu_{t_0^1}^1 + \int \log J_1 \mu_{t_0^1}^1=  h(\mu_{t_0^1}^1) + \int \log J_1 \mu_{t_0^1}^1.
\end{equation}

\section{A Bayesian Hypotheses Test and a Chernoff's Information Version} \label{Bayes}
\renewcommand{\theequation}{\thesection.\arabic{equation}}
\setcounter{equation}{0}

In this section, we are interested in finding a version of  Chernoff's information for the setting of Thermodynamic Formalism (see theorem 11.9.1 in \cite{Cover}).

Given two H\"older Jacobians $J_0$ and $J_1$, for a given parameter $0 \leq \lambda \leq 1$, consider their convex combination given by

\begin{equation} \label{sac}J_\lambda = \lambda J_ 1 + (1-\lambda) J_0.
\end{equation}
We point out that $J_\lambda$ is also a H\"older Jacobian and our setting has a different nature than the one mentioned in section 10 of \cite{Cover}. Making an analogy, for being consistent with \cite{Cover},  we should consider a convex combination of the logarithm of the Jacobians $J_0$ and $J_1$; instead, we did not use the logarithm function. In \cite{Cover} the probabilities are on finite sets.

We denote by $\mu_\lambda$ the H\"older Gibbs probability associated with $\log J_\lambda$.
For $\lambda\in  [0,1]$ and $n \in \mathbb{N}$, set $S_n^\lambda (x)$ as
\begin{equation*}
S_n^\lambda (x):=\,\,\frac{1}{n} \sum_{i=0}^{n-1} \log\left(J_\lambda(\sigma^i (y)) \right).
\end{equation*}

The \emph{rejections regions} are of the form
\begin{equation} \label{yokono1}
\mathcal{R}_{n,\lambda}:=\left\{x\in \Omega \,|\, \frac{1}{n} \sum_{i=0}^{n-1} \log J_{\lambda}(\sigma^i (y))   < u_n\right\}, \ \ \mbox{ for } \ \ n\in \mathbb{N}.
\end{equation}

The \emph{a priori probability} of $H_0$ is given by $\pi_0$ and the \emph{a priori probability} of $H_1$ is given by $\pi_1= 1 -\pi_0$. We shall consider for the Bayesian hypotheses test a sequence $u_n \to E$, $n \in \mathbb{N}$, in the same way as in \eqref{yuo}. The expression
\begin{equation*}
\pi_1\ \mu_1 \{x\,|\, S_n^\lambda>  u_n\}
\end{equation*}
represents the mean value probability of a wrong decision by announcing $H_0$ when $H_1$ is true while the expression
\begin{equation*}
\pi_0\,  \mu_0 \{x\,|\, S_n^\lambda\leq  u_n\},
\end{equation*}
represents the mean value probability of a wrong decision by announcing $H_1$ when $H_0$ is true.

In the Bayes hypotheses test, for each value of $\lambda\in[0,1]$, we want to minimize the average total mean probability. We want to choose $u_n$, $n \in \mathbb{N}$, this means to choose $E$, that asymptotically minimizes
\begin{equation} \label{zcafe}
\pi_0\, \mu_0 \{x\,|\, S_n^\lambda\leq  u_n\} + \pi_1\, \mu_1 \{x\,|\, S_n^\lambda>  u_n\}, \ \ \mbox{as} \ \  n \to \infty.
\end{equation}

We shall denote by $E_\lambda$ the best value of $E$, for each value $\lambda$. We will show later the explicit expression for such $E_\lambda$. We can also ask: among the  different values of $\lambda$ which one determines the best $E_\lambda$, in the sense of getting the best rate? We shall denote by $\tilde{E},\tilde{\lambda}$ the optimal value, among all possible values of $E$ and $\lambda$, for the asymptotic \eqref{zcafe} minimizer.
In our reasoning, we want to find the best choice of $\tilde{\lambda}$ for which a best choice of $\tilde{E}$ is possible.

In the same way as before, it will follow that, for each choice of $\lambda$ and limit value $E$, we shall obtain $C_1(E,\lambda)=C_1\geq 0$ and $C_0(E,\lambda)=C_0\geq 0$, such that,
\begin{equation} \label{zcafe2}
\mu_1 \{x\,|\, S_n^\lambda>  u_n\}\sim \,e^{-C_1\,   n} \ \ \text{ and }\ \
\mu_0 \{x\,|\, S_n^\lambda\leq   u_n\}\sim \, e^{-C_0\, n}\,.
\end{equation}

From the expression \eqref{zcafe2}, for each $E$ and $\lambda$ we get that
the asymptotic of \eqref{zcafe} is of order
\begin{equation} \label{zcafe1}
e^{- \, \min\{C_0(E,\lambda),C_1(E,\lambda)\}\, n}.
\end{equation}

When $C_0(E,\lambda)=0$ or $C_1(E,\lambda)=0$ we do not get the optimal values of $E$ and $\lambda$ for the asymptotic in \eqref{zcafe}. Such values of $E$ and $\lambda$ should be discarded. We will show that for the optimal solution it is required that $C_0(E,\lambda)=C_1(E,\lambda)$. This optimal solution is called \emph{Chernoff information} (we refer the reader to the end of the proof of theorem 11.9.1 in \cite{Cover}, which
considers a different setting).

The optimal choice of $E$ and $\lambda$ will be described by expressions
\eqref{cross}, \eqref{relentropy} and \eqref{xcafe4} at the end of this section. The optimal value $C_0(E,\lambda)$ will have a relative entropy expression given by \eqref{relentropy}.

We shall be interested in estimating
\begin{align*}
&\mu_{1}\left( S_n^\lambda> u_n \right)= \mu_{1}\left( S_n^\lambda - u_n >0 \right),\ \
\mbox{ and also } \ \
\mu_{0}\left( S_n^\lambda\leq u_n \right)= \mu_{0}\left( S_n^\lambda - u_n \leq 0 \right),
\end{align*}
where $u_n \to E$. This requires to estimate
\begin{equation*}
\mu_{j} (\,(\, S_n^\lambda - u_n\,) \in (a,b)\,)\,=\mathbb{P}_{\mu_{j}}  \left( \frac{1}{n} \sum_{i=0}^{n-1}  \left[\log\left(J_{\lambda}(\sigma^i (y)) \right)- u_n \right]\in (a,b)  \right),
\end{equation*}
for $j=0,1$.

In order to get the correct large deviation rate we need first to analyze the
following expression
\begin{equation*}
\phi_{n,\lambda} ^j(t):=\frac{1}{n}\log\left( \int e^{\,  t \sum_{i=1}^{n}\,\log J_{\lambda}(\sigma^i (y)) }\mathrm{d} \mu_{j}(y) \right) - t  \,u_n,
\end{equation*}
for each $n$, $\lambda$ and real value $t$.

For $j=0,1$, $\lambda\in[0,1]$ and $t \in \mathbb{R}$,
\begin{equation*}
\lim_{n \to \infty}   \frac{1}{n}\log\left( \int e^{\,  t \sum_{i=1}^{n} \log J_{\lambda}(\sigma^i (y)) }\mathrm{d} \mu_{j}(y) \right)= P( t \,\log J_{\lambda}  + \log J_{j\,}).
\end{equation*}

Then, for $j=0,1$, $\lambda\in[0,1]$ and $t \in \mathbb{R}$, denote
\begin{equation*}
\phi_\lambda^j (t) :=  \lim_{n \to \infty} \phi_{n,\lambda}^j(t) =P( t \,\log J_{\lambda}  + \log J_{j\,}) - t \, E.
\end{equation*}
We denote by  $P_{j,\lambda} $, for $j=0,1$ and $\lambda\in[0,1]$, the function
\begin{equation*}
t \to P_{j,\lambda} (t)=P( t \,\log J_{\lambda} + \log J_{j\,}),
\end{equation*}
which is convex and also monotone decreasing in $t$. Moreover, $P_{j,\lambda}(0)=P( 0  \,\log J_{\lambda} + \log J_{j\,})=0$, for $j=0,1$ and for $\lambda\in[0,1]$.

Note that
\begin{equation} \label{krac}
\frac{d}{dt}P_{j,\lambda}(t)|_{t=0}= \frac{d}{dt}\, P( t \log J_{\lambda}  + \log J_{j\,})|_{t=0}= \int \log J_{\lambda}  d \mu_{j}.
\end{equation}
Furthermore, for $t \in \mathbb{R}$,
\begin{equation} \label{rer}
\frac{d}{dt}\, P( t \,\log J_{\lambda} + \log J_{j\,})|_{t}= \int \log J_{\lambda}  d \mu^{t,\lambda}_j <0,
\end{equation}
where $\mu^{t,\lambda}_j$ is the equilibrium probability for $t\log J_{\lambda} + \log J_{j\,}$.



The deviation function $I_j^\lambda$ for $(S_n^\lambda - u_n)$, $n \in \mathbb{N}$ and
for $\mu_j$, $j=0,1$, is
\begin{align}\label{dadai}
\nonumber
I_j^\lambda(x)& =\sup_t\left[t x -\phi_\lambda^j(t) \right]=
\sup_t\left[ t\left( x  + E\, \right)-  P(t \,\log J_{\lambda}  + \log J_{j\,}) \right] \\
&= \sup_t\left[ t\left( x  + E\, \right)-  P_{j,\lambda}( t ) \right].
\end{align}
If
\begin{equation*}
x=v_j^\lambda = v_j^{E,\lambda}= - E      \,+  \,\int \log J_{\lambda\,}  d \mu_{j},
\end{equation*}
then, $t=0$ and,
\begin{equation} \label{tri10}
I_j^\lambda(v_j^{E,\lambda})=0.
\end{equation}
The suitable values $v_j^{E,\lambda}$ are the ones such that
$v_1^{E,\lambda}<0$ and $v_0^{E,\lambda}>0$. For each fixed $\lambda$, this will require that
\begin{equation*}
E\geq   \,\int \log J_{\lambda\,}  d \mu_{1} \ \ \mbox{ and } \ \
E\leq   \,\int \log J_{\lambda\,}  d \mu_{0}.
\end{equation*}
We will show there exist values $\lambda$ such that it is possible to find a non-trivial interval for $E$. We just have to find values $\lambda$, such that
\begin{equation} \label{baba}
\int \log J_{\lambda\,}  d \mu_{0}> \int \log J_{\lambda\,}  d \mu_{1}.
\end{equation}
We claim that there are values of $\lambda$ such that the expression
\eqref{baba} holds. Indeed,
\begin{align*}
& \lambda=0 \Rightarrow \int \log J_{0\,}d \mu_{0} - \int \log J_{0\,}d \mu_{1}>0, \, \mbox{while} \,
\lambda=1 \Rightarrow \int \log J_{1\,}d \mu_{0} - \int \log J_{1\,}d \mu_{1}<0.
\end{align*}
There exists a value $\lambda$ such that
\begin{equation} \label{KLM}
\int \log J_{\lambda\,}  d \mu_{0} - \int \log J_{\lambda\,}  d \mu_{1}=0,
\end{equation}
that is, there exists a value $\lambda$ such that
\begin{equation*}
\int  \log (\lambda J_ 1 + (1-\lambda) J_0) d\mu_0= \int \log ( \lambda J_ 1 + (1-\lambda) J_0) d\mu_1.
\end{equation*}
In fact, consider the functions
\begin{equation*}
g_j (\lambda) = \int \log (\lambda J_ 1 + (1-\lambda) J_0) d\mu_j,
\end{equation*}
for $j=0,1$.

Note that $ g_0(0) > g_1(0) $ and $g_0 (1)< g_1(1)$.
From this fact, the claim follows. The functions $g_j$, for $j=0,1$, are concave and
$g_0$ is \emph{a decreasing function} while $g_1$ is \emph{an increasing} one. Besides,
the point where the two graphs coincide is unique.

Therefore, \emph{there exists a value $\lambda_s$}, such that for $0\leq \lambda<\lambda_s$, a non trivial interval of suitable parameters $E$ exists and it holds that
\begin{equation}\label{baba1}
0 >\int \log J_{\lambda\,}  d \mu_{0}>E > \int \log J_{\lambda\,}  d \mu_{1}.
\end{equation}
In this case, for such parameters $E$, we have $v_1^{E,\lambda}<0$ and $v_0^{E,\lambda}>0$.
From now on we assume that $E$ is in the interval described by the expression \eqref{baba1}.


When $x=0$, for $\lambda \in [0,\lambda_s]$, for $j=0,1$, we get $t^{E,\lambda}_j \in \mathbb{R}$, for which
\begin{equation} \label{caca1}
P_{j,\lambda}^{\prime } (t^{E,\lambda}_j )=\frac{d}{dt}\, P( t \log J_{\lambda}  + \log J_{j\,})|_{t^{E,\lambda}_j }=E=\int \log J_{\lambda}  d \mu_{t^{E,\lambda}_j },
\end{equation}
where  $\mu_{t^{E,\lambda}_j }$ is the equilibrium probability for $t^{E,\lambda}_j  \log J_{\lambda} + \log J_{j\,}$, where  $E$ satisfies \eqref{baba1}.

From the convexity argument and expression \eqref{rer}, for fixed $\lambda$ and for $j=0,1$, the value $t^{E,\lambda}_j $ is monotonous increasing on $E$. That is, for fixed $\lambda$ and for $j=0,1$, the function $E \to t^{E,\lambda}_j$ satisfies
\begin{equation} \label{tri113}
\frac{d}{d E} t^{E,\lambda}_j \,  >0.
\end{equation}
Denote $I_j^{E,\lambda}(0)$ by
\begin{align}\label{tri112}
\nonumber
I_j^{E,\lambda}(0) \,:\,& =t^{E,\lambda}_j  E-      P( t^{E,\lambda}_j  \log J_{\lambda} + \log J_{j\,})  = t^{E,\lambda}_j  E- P_{j,\lambda}(t^{E,\lambda}_j ) \\
& - \left[\int \log J_j d \mu_{t^{E,\lambda}_j }+ h(\mu_{t^{E,\lambda}_j})\right]=
- \left[\int \log J_j d \mu_{t^{E,\lambda}_j } - \int \log J_{j,\lambda,E  } (\mu_{t^{E,\lambda}_j})\right]>0,
\end{align}
where $J_{j,\lambda,E  }$ is the Jacobian of the invariant  probability $\mu_{t^{E,\lambda}_j }$ which is, by it turns, the equilibrium probability for    $t^{E,\lambda}_j  \log J_{\lambda} + \log J_{j\,}$, for $j=0,1$ and $\lambda \in [0,\lambda_s]$.

Using expressions \eqref{tri112}  and \eqref{dadai}, when $u_n \to E$,  one can rewrite them both, as mentioned in \eqref{zcafe}, by
\begin{equation} \label{tri200}
\mu_j \{x\,|\, S_n^\lambda>  u_n\}\sim \,e^{-\,  I_j^{E,\lambda}(0)  n},
\end{equation}
for $j=0,1$. Hence, in the notation of \eqref{zcafe2}, we get  $C_j(\lambda,E) = I_j^{E,\lambda}(0)$, for $j=0,1$.

Furthermore, for fixed $\lambda$ and $j=0,1$,
\begin{equation} \label{tri111}
\frac{d}{d E} \left[t^{E,\lambda}_j \,   P_{j,\lambda}^{\prime} (
t^{E,\lambda}_j )\,-\, P_{j,\lambda}(t^{E,\lambda}_j )\right]= \,(  t^{E,\lambda}_j )\, (  t^{E,\lambda}_j )^{\prime}\, P_{j,\lambda}^{\prime \prime} ( t^{E,\lambda}_j   ).
\end{equation}
Since
\begin{equation*}
P( 0  \,\log J_{\lambda} + \log J_{0\,})=0,
\end{equation*}
\begin{equation*}
\frac{d}{dt} P_{0,\lambda}(t)|_{t=0} = \int \log J_{\lambda}  d \mu_{0} \ \ \mbox{ and }
\ \ 0 >\int \log J_{\lambda\,}  d \mu_{0}>E.
\end{equation*}
From the pressure convexity, we get $t^{E,\lambda}_0<0$.

\begin{figure}[!htb]
	\centering
	{\includegraphics[scale=0.5]{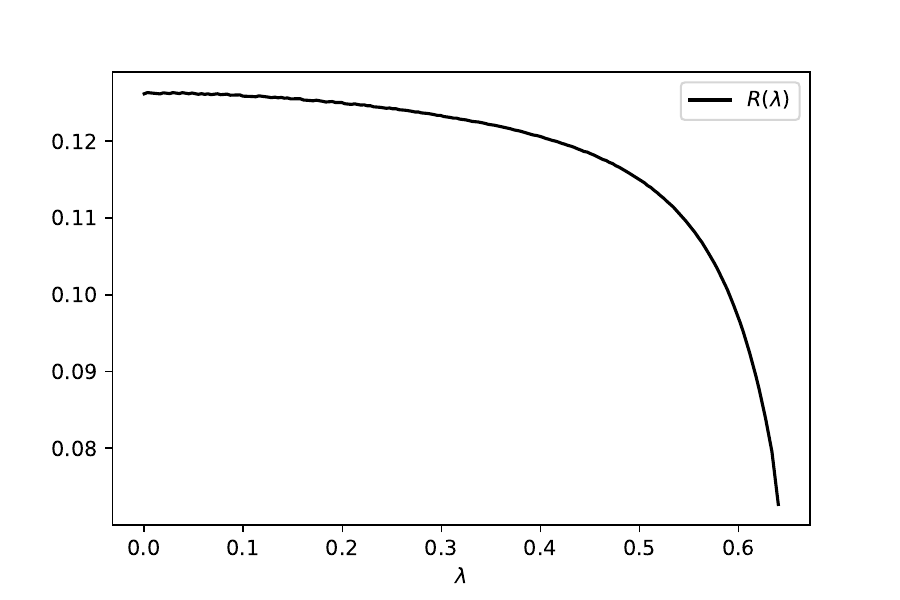}}
\caption{Graph of the function $R(\lambda) = I_0^{E_{\lambda, \lambda}}(0)$, when $0\leq \lambda\leq \lambda_s$, using the stochastic matrix $\mathcal{P}_{j}$, for $j=0,1$, from the example in Section \ref{Example}.}
\label{Figure_6.1}
\end{figure}

From expressions \eqref{tri112}, \eqref{tri111} and \eqref{tri113}, it follows that
we get
\begin{equation} \label{jke1}
E\,\to\, I_0^{E,\lambda}(0)\ \ \mbox{ decreases with }\ \ E,
\end{equation}
for each fixed $\lambda$.
As  $\int \log J_{\lambda\,}  d \mu_{1}<E$, we  obtain $  t^{E,\lambda}_ 1>0$.
From this property, one can show, in a similar way, that
\begin{equation} \label{jke2}
E\,\to\, I_1^{E,\lambda}(0) \ \ \mbox{ increases with }\ \ E,
\end{equation}
for each fixed $\lambda$.

For fixed $0\leq \lambda\leq \lambda_s$, consider the functions
\begin{equation*}
E\in \left[\int \log J_\lambda d \mu_1 ,\int \log J_\lambda d \mu_0  \right]  \to y_a(E)= I_0^{E,\lambda}(0)
\end{equation*}
and
\begin{equation*}
E\in \left[\int \log J_\lambda d \mu_1 ,\int \log J_\lambda d \mu_0 \right]  \to y_b(E)= I_1^{E,\lambda}(0).
\end{equation*}

As $y_a(  \int \log J_\lambda d \mu_0)=0 $, it follows, from the decreasing  monotonicity  (see expression \eqref{jke1}), that $y_a(  \int \log J_\lambda d \mu_1)>0$. In fact,
$\int \log J_\lambda d \mu_0 - \int \log J_\lambda d \mu_1>0$.

Moreover, as $y_b(  \int \log J_\lambda d \mu_1)=0 $, from the increasing monotonicity (see expression \eqref{jke2}), we get $y_b(  \int \log J_\lambda d \mu_0)>0$. Hence, \emph{for each} $\lambda$, $0\leq \lambda\leq \lambda_s$, \emph{there exists a point} $E_\lambda$, such that $I_1^{E_\lambda,\lambda}(0)=I_0^{E_\lambda,\lambda}(0)$. The value  $E_\lambda$ \emph{determines the best rate for the parameter} $\lambda$.

Note also that this point $E_\lambda$  belongs to the interval  $[\int \log J_\lambda d \mu_1 ,\int \log J_\lambda d \mu_0  ] $. Furthermore, if $\lambda_p<\lambda_q\leq \lambda_s$, then, as $g_0$ is a decreasing function and $g_1$ is an increasing one,
we have that  $[\int \log J_{\lambda_q} d \mu_1 ,\int \log J_{\lambda_q} d \mu_0  ] \subset [\int \log J_{\lambda_p} d \mu_1 ,\int \log J_{\lambda_p} d \mu_0  ] $.
Moreover, $E_{\lambda_s}=\int \log J_{\lambda_s} d \mu_1=\int \log J_{\lambda_s} d \mu_0 $.

Note that, for any $0\leq \lambda<\lambda_s$, the value
\begin{equation*}
\int \log J_{\lambda_s} \,d\,  \mu_{t^{E_{\lambda_s} ,\lambda_s}_1 }     \in \left[\int \log J_{\lambda} d \mu_1 ,\int \log J_{\lambda} d \mu_0  \right].
\end{equation*}

\begin{figure}[!htb]
	\centering
	{\includegraphics[scale=0.5]{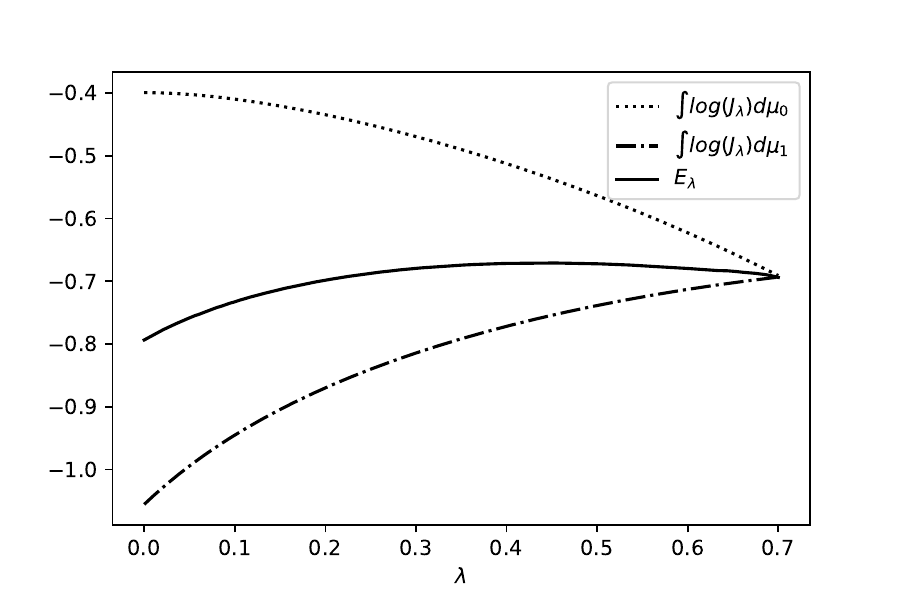}}
\caption{Graphs of the functions $\lambda \to \int \log J_\lambda d \mu_0$ (in dotted line) and
$\lambda \to \int \log J_\lambda d \mu_1$ (in dashed and dotted line) together with the graph of the
values $E_\lambda$ (in solid line), as a function of $\lambda$, when $0\leq \lambda\leq \lambda_s$. The stochastic matrix $\mathcal{P}_{j}$, for $j=0,1$, is from the example in Section \ref{Example}.}
\label{Figure_6.2}
\end{figure}

From the expression \eqref{zcafe1},  property \eqref{tri200} and the fact that $I_0^{E,\lambda}(0)$ decreases with $E$ (while $I_1^{E,\lambda}(0)$ increases with $E$), for each  fixed $\lambda$, we get the best value $E$ is when $E=E_\lambda$ (see definition above). That is, when
\begin{equation} \label{xcafe}
t^{E_\lambda,\lambda}_0 E_\lambda- P_{0,\lambda}(t^{E_\lambda,\lambda}_0) =  I_0^{E_\lambda,\lambda}(0) = I_1^{E_\lambda, \lambda}(0) =   t^{E_\lambda,\lambda}_1 E_\lambda- P_{1,\lambda}(t^{E_\lambda,\lambda}_1).
\end{equation}

Then, we need to find the value $\tilde{\lambda}$, which maximizes $\lambda \to I_1^{E_\lambda, \lambda}(0)$, simultaneously with \eqref{xcafe} holds, among all $\lambda$, such that  $\lambda \in [0,\lambda_s]$.

Note that when $\lambda=\lambda_s$, we have both  $I_0^{E_\lambda, \lambda}(0) = 0= I_1^{E_\lambda, \lambda}(0)$, which is not a good choice.

The function $\lambda \to I_0^{E_{\lambda, \lambda}}(0) $ is monotonous decreasing on $\lambda\in [0,\lambda_s]$. Therefore, the largest $I_0^{E_{\lambda},\lambda}(0)$ occurs
when $\lambda=0$. This means that we need to take $J_\lambda=J_0$.

The value $ E_0$ belongs to the interval $[\int \log J_0 d \mu_1 ,\int \log J_0 d \mu_0  ]$ and, by definition,  $I_0^{E_0,0}(0) = I_1^{E_0,0}(0)$.

The value $E_0$ is determined (see expression \eqref{caca1}) by the equation
\begin{equation} \label{cross}
P_{0,0}^{\prime }  (t_0^{E_0,0})=E_0= P_{1,0}^{\prime }  (t_1^{E_0,0}).
\end{equation}



From expression \eqref{tri112}, the corresponding value of
$I_0^{E_0,0}(0) $ is given by
\begin{equation}\label{relentropy}
I_0^{E_{0},0}(0) =   t^{E_0,0}_0  E_0- P_{0,0}(t^{E_0,0}_0 )=
- \left [\int \log J_0\, d \mu_{t^{E_0,0}_0 }- \int \log J_{0,0,E_0  }\,d \,\mu_{t^{E_0,0}_0 }\right]>0.
\end{equation}


Expression \eqref{relentropy} is a relative entropy.

Finally, the best choice for the hypotheses test, under thermodynamic formalism sense,
will be when the rejection region is of the form
\begin{equation} \label{yokono134}
\mathcal{R}_{n,0}=\left \{y\in \Omega \,\bigg|\, \frac{1}{n} \sum_{i=0}^{n-1} \log J_{0}(\sigma^i (y))   < u_n\right \},\ \ n \in \mathbb{N},
\end{equation}
with $u_n \to E_0$, and $E_0$ satisfies \eqref{cross}.

In this case,
\begin{equation} \label{xcafe4} \pi_0\, \mu_0
\{x\,|\, S_n^0\leq  u_n\} + \pi_1\, \mu_1 \{x\,|\, S_n^0>  u_n\}\,\,\sim e^{-  \,I_0^{E_{0},0}(0)\,n} \end{equation}
will describe the  best possible rate among $\lambda$ for minimizing the probability of a wrong decision.

\section{An Example} \label{Example}
\renewcommand{\theequation}{\thesection.\arabic{equation}}
\setcounter{equation}{0}

In this section, we present an example for the Min-Max Hypotheses test.
Recall that, given the two by two  line stochastic matrix $\mathcal{P}$, the value of the Jacobian $J$ on the cylinder $\overline{i\,j}$ has the constant value $\frac{\pi_i\,p_{i j}}{\pi_j}= p_{ji}$, where $\pi=(\pi_1,\pi_2)$ is the initial stationary vector for  $\mathcal{P}$.

We consider the case where $\mathcal{P}_0$ and $\mathcal{P}_1$ are described by the following two column stochastic matrices
\[
\mathcal{P}_0:= \begin{pmatrix}
\frac{1}{4} &\frac{1}{2} \\
\frac{3}{4} &\frac{1}{2}
\end{pmatrix}
\hspace{0.3cm}\mbox{ and }\hspace{0.3cm}
\mathcal{P}_1:= \begin{pmatrix}
\frac{2}{3} &\frac{1}{5} \\
\frac{1}{3} &\frac{4}{5}
\end{pmatrix}.
\]

In this case, the best rate is described by \eqref{mmi17} as it
was explained in Section \ref{MinMax}. We shall present now some explicit values for this
example.

Using techniques given in \cite{FLO} one can show that
the maximizing probability (see \cite{BLL})  for $\log J_0 - \log J_1$ is  an orbit of period two. More precisely, $m(\log J_0-\log J_1)= \frac12 \log\left(\frac{45}{8}\right) $ which is a value close to $c^{+} \sim 0.8636$. For $m(\log J_1-\log J_0)$, which is realized by a orbit of period $1$,  we get $c^{-}\sim  0.9808$.

In fact, we are interested in finding the image of the functions $P'_0$ and $P'_1$, that is, in finding $c^{+}$ and $c^{-}$. We will show the domain of the Legendre transform for $P_1$, which is the same as for $P_0$.

Define $K = \log(J_0/J_1)$. It is true that $\lim_{t \rightarrow +\infty}P'_0(t)= \lim_{t \rightarrow +\infty}P'_1(t)$, and  this limit  value (see \cite{BLL}) is given by $m(K)$, as long as exists a function $u$, the so-called a \emph{calibrated subaction}, such that
\begin{equation*}
\max_{y}[ K(yx)+u(yx)] = m(K) + u(x),
\end{equation*}
for all $x\in \Omega$. We claim that this equation is satisfied when
\begin{equation*}
m(K)= \frac{1}{2}\log\left(\frac{45}{8}\right).
\end{equation*}
We refer the reader to the reference \cite {BLL} for the max-plus algebra's properties in Ergodic Optimization. The proof of the claim will be as follows. Define the matrix, with $\varepsilon = -\infty$, by
\[W:= \begin{pmatrix}
K_{11} &\varepsilon & K_{21} &  \varepsilon\\
K_{11} &\varepsilon & K_{21} &  \varepsilon \\
\varepsilon &K_{12}  &  \varepsilon & K_{22} \\
\varepsilon &K_{12}  &  \varepsilon & K_{22}
\end{pmatrix}.
\]
Now,
\[
m(K)= \bigoplus_{n=1}^4 \frac{Tr_{\oplus}(W^{\otimes n})}{n}
\]
is simply the maximum cyclic mean in the directed graph which has transition costs $W_{ij}$ from node $i$ to node $j$. Here, we denote $Tr_{\oplus}$ the max-plus trace and $W^{\otimes n}$ is the $n$-th max-plus power of $W$. It is easy to see that the maximal cyclic mean in such graph is given by the mean $\frac{W_{23}+W_{32}}{2} = \frac{1}{2}\log\left(\frac{45}{8}\right)$. In fact, from this the following matrix
\[u = \begin{pmatrix}
K_{21}+K_{12}-2m(K) &0 \\
 K_{12} -m(K)& K_{12}-m(K)
\end{pmatrix}
\]
is a calibrated sub-action, that is, the calibrated subaction is determined by the matrix $u$. Hence, we conclude that  $\lim_{t \rightarrow +\infty}P'_{i=1,2}(t)= m(K)=\frac{1}{2}\log\left(\frac{45}{8}\right)$.

We can also compute $\lim_{t \rightarrow - \infty}P_{i=1,2}'(t)=-m(-K)$ by following the same procedure but now with the matrix $-W$ (here we only change the sign of the finite terms in the matrix $W$). In this way, we consider
\[
m(-K)=\log\left(\frac{8}{3}\right) \ \ \mbox{ and } \ \   v =  \begin{pmatrix}
0 &0\\
-K_{12}-m(-K)&-K_{12}-m(-K)
\end{pmatrix} ,
\]
which satisfies $\max_{Ty=x}\{-K(y)+v(y)\}=m(-K)+v(x)$. This means that
\begin{equation*}
\lim_{t \rightarrow -\infty} P'_{i=1,2}(t)= -\log\left(\frac{8}{3}\right).
\end{equation*}
This concludes the example.
$\hfill \diamondsuit$

The method described in this section can be adapted to other cases.

\subsection*{Acknowledgments}

H.H. Ferreira was supported by CAPES-Brazil.
A.O. Lopes' and S.R.C. Lopes' research were partially supported by CNPq-Brazil.



\begin{thebibliography}{1}
\footnotesize{

\bibitem{AR}
F. Abramovich and Y. Ritov,
\emph{Statistical Theory: A Concise Introduction}, Boca Raton, CRC Press, 2013.



\bibitem{Baha}
R. R. Bahadur. Large deviations of the maximum likelihood estimate
in the Markov chain case. In J. S. Rostag, M. H. Rizvi and D. Siegmund,
editors, Recent advances in statistics, 273–-283.  Boston, Academic
Press, 1983.


\bibitem{BLL} A. Baraviera, R. Leplaideur and A. O. Lopes, \emph{Ergodic Optimization, zero temperature and the Max-Plus algebra}. $23^{\text{o}}$ Col\'oquio Brasileiro de Matem\'atica,  IMPA, Rio de Janeiro, 2013.

\bibitem{Barni}  M. Barni, B. Tondi,
Binary Hypothesis Testing Game With Training Data,
IEEE Transactions on Information Theory (Volume: 60, Issue: 8, Aug. 2014)

\bibitem{Beno} T. Benoist, V. Jaksic , Y. Pautrat, C.-A. Pillet,  On Entropy Production of Repeated Quantum
Measurements I. General Theory. \emph{Commun. Math. Phys.}, Vol. 357, 77--123, 2018.


\bibitem{Boh} D. Bohle, A. Marynych and M. Meiners, A Fundamental Problem of Hypotesis testing with finite e-commerce, arXiv, 2020.

\bibitem{Bla} D. Blackwell and M. A. Girshick,
Theory of Games and Statistical Decisions, Dover publications (1979)

\bibitem{Broe} L. D. Broemeling,
\emph{Bayesian Inference
for Stochastic Processes}, Boca Raton, CRC Press, 2018.



\bibitem{Buck} J. A. Bucklew, \emph{Large Deviation Techniques in Decision, Simulation and Estimation}. New York, Wiley, 1990.


\bibitem{Cat} A. Caticha, Entropic Physics: Lectures on Probability, Entropy and Statistical Physics, arXiv, 2021.


\bibitem{Cha}  J-R Chazottes and D Gabrielli,  Large deviations for empirical entropies of g-measures. \emph{Nonlinearity}, Vol. 18, 2545--2563, 2005.

\bibitem{Cover}  T. Cover and J. Thomas, \emph{Elements of Information Theory}, second edition. New York, Wiley Press, 2006.


\bibitem{CLP}  G.B. Cybis, S. R. C. Lopes and  H.P. Pinheiro,
Power of the Likelihood Ratio Test for Models of DNA Base Substitution. \emph{Journal of Applied Statistics}, Vol. 38, 2723--2737, 2011.



\bibitem{Denk2} R. Dakovic, M. Denker and M. Gordin,
Circular unitary ensembles: parametric models and their asymptotic maximum likelihood estimates. \emph{Journal of Mathematical Sciences}, Vol. 219(5), 714--730, 2016.



\bibitem{Denk} M. Denker and W. Woyczynski. \emph{Introductory Statistics and Random Phenomena: Uncertainty, Complexity and Chaotic Behavior in Engineering and Science}. New York, Birkh\"auser, 2012.


\bibitem{Denk1} M. Denker,
Basics of Thermodynamics, Lecture Notes - Penn State Univ., 2011.



\bibitem{DZ} A. Dembo and O. Zeitouni, \emph{Large Deviation Techniques and Applications}. New York, Springer Verlag, 2010.


\bibitem{Ellis} R. Ellis, \emph{Entropy, Large Deviations, and Statistical Mechanics}, New York, Springer Verlag, 2005

\bibitem{ELLM} A.C.D. van Enter, A. O. Lopes, S. R. C Lopes and J. K. Mengue,
How to get the Bayesian a posteriori
probability from an a priori probability via
Thermodynamic Formalism for plans;
the connection to Disordered Systems. preprint, 2021.


\bibitem{FLO}  H.H. Ferreira, A.O. Lopes and E.R. Oliveira,
An iteration process for approximating subactions. To appear in ``Modeling, Dynamics, Optimization and Bioeconomics IV" Editors: Alberto Pinto and David Zilberman, Springer Proceedings in Mathematics and Statistics, New York, Springer Verlag, 2021.


\bibitem{GR} V. Girardin and P. Regnault, Escort distributions minimizing the Kullback–Leibler divergence for a large
deviations principle and tests of entropy level. \emph{Ann Inst Stat Math.}, Vol. 68, 439–-468, 2016.

\bibitem{GLR} V. Girardin,  L. Lhote and P. Regnault,
Different Closed-Form Expressions for Generalized Entropy
Rates of Markov Chains. \emph{Methodology and Computing in Applied Probability}, Vol. 21, 1431--1452, 2019.








\bibitem{KLS} M.J. Karling, S.R.C. Lopes and R.M. de Souza, A Bayesian Approach for Estimating the Parameters of an $\alpha$-Stable Distribution. emph{Journal of Statistical Computation and Simulation},  Volume 91, 2021 - Issue 9 (2021) 1713-1748.




\bibitem{Kif}
Y. Kifer, Large Deviations in Dynamical Systems and Stochastic processes, \emph{Trans. Amer. Math. Soc.}, Vol. 321(2), 505--524, 1990.



\bibitem{LLV} A.O. Lopes, S. R. C. Lopes and P. Varandas,
Bayes posterior convergence for loss functions via almost additive Thermodynamic Formalism, to appear in Journ. of Statis. Physics

\bibitem{LM} A.O. Lopes and J.K. Mengue,
On information gain, Kullback-Leibler divergence, entropy production and the involution kernel, to appear in Disc. and Cont. Dyn. Syst. Series A

\bibitem{LR} A.O. Lopes and R. Ruggiero,
Nonequilibrium in Thermodynamic Formalism: the Second Law, gases and Information Geometry, Qualitative Theory of Dynamical Systems, 21-21 p1-44 (2022)

\bibitem{L4} A.O. Lopes,
Entropy, Pressure and Large Deviation, \emph{ Cellular Automata, Dynamical Systems and Neural Networks}, E. Goles e S. Martinez (eds.), Kluwer, Massachusets, 79--146, 1994.

\bibitem{L2} A.O. Lopes, Entropy and Large Deviation,
\emph{NonLinearity}, Vol. 3(2), 527--546, 1990.


\bibitem{L3} A.O. Lopes,
Thermodynamic Formalism, Maximizing Probabilities and Large Deviations, Preprint - UFRGS.



\bibitem{Nobel1} K. McGoff, S. Mukherjee and A. Nobel, Gibbs posterior convergence and Thermodynamic formalism, to appear in Adv. in Appl. Prob.


\bibitem{PP}
W. Parry and M.  Pollicott. Zeta functions and the periodic
orbit structure of hyperbolic dynamics. \emph{Ast\'erisque},
Vol. 187-188, 1990.

\bibitem{Roh} V. K. Rohatgi,
\emph{An Introduction to Probability Theory and Mathematical Statistics}. New York: Wiley, 1976.

\bibitem{Sagawa} T. Sagawa, Entropy, divergence and majorization in classical and quantum theory, arXiv, 2020.



\bibitem{Suhov}
Y. Suhov and M. Kelbert,
\emph{Probability and Statistics by Example. II}, Cambridge, Cambridge University Press, 2014.

\bibitem{van} D. A. van Dyk, The Role of Statistics in the Discovery of a Higgs Boson. Annual Review of Statistics and Its Application. 1 (1): 41–59 (2014)

\bibitem{von} W. von der Linden, V. Dose and U. von Toussaint, \emph{Bayesian Probability Theory Applications in the Physical Sciences}. Cambridge, Cambridge University Press, 2014.


}

\end{thebibliography}
\end{document}